\documentclass[11pt]{article}
\usepackage{graphicx}
\usepackage{amsmath}
\usepackage{amssymb}
\usepackage{amsfonts}
\usepackage{xspace}
\usepackage{subfigure}
\usepackage{mathtools}
\usepackage{bm}
\usepackage{stmaryrd}
\usepackage[colorlinks=true,
            linkcolor=blue,
            urlcolor=blue,
            citecolor=blue]{hyperref}
\usepackage{braket}
\usepackage{cleveref}
\usepackage{multirow}
\usepackage{xcolor}
\usepackage{comment}
\usepackage{stackrel}
\usepackage{enumitem}

\usepackage[utf8]{inputenc}
\usepackage{pgfplots}
\pgfplotsset{compat=newest}
\usepgfplotslibrary{groupplots}
\usepgfplotslibrary{dateplot}
\usepackage{nicefrac}
\usepackage{booktabs}

\usepackage{tcolorbox}
\usepackage{mathabx}

\crefname{equation}{}{}
\crefname{table}{Table}{Tables}
\crefname{figure}{Figure}{Figures}
\crefname{section}{Section}{Sections}

\crefname{theorem}{Theorem}{Theorems}
\crefname{remark}{Remark}{Remarks}
\crefname{lemma}{Lemma}{Lemmas}
\crefname{proposition}{Proposition}{Propositions}
\crefname{definition}{Definition}{Definitions}
\crefname{observation}{Observation}{Observations}

\newcommand{\x}{\mathring{x}}

\usepackage{todonotes}

\newcommand{\triang}{\mathcal{T}}

\makeatletter
\newcommand{\fcdot}{\,\cdot\,}
\newcommand{\fcarg}[1]{\def\fc@rg{#1}\ifx\fc@rg\empty\fcdot\else\fc@rg\fi}
\makeatother
\newcommand{\abs}[1]{\lvert\fcarg{#1}\rvert}

\newcommand{\field}{\mathbb}
\newcommand{\reals}{\field{R}}
\newcommand{\R}{\reals} %
\newcommand{\N}{\field{N}} %

\newcommand{\TT}{\mathcal{T}}

\newcommand{\node}{n}

\newcommand{\abbr}[1][abbrev]{#1.\ }%

\newcommand{\ie}{\abbr[i.e]}

\newcommand{\wrt}{\abbr[w.r.t]}

\newcommand{\pwl}{\abbr[pwl]}

\newcommand{\Pwl}{\abbr[Pwl]}

\makeatletter
\newcommand{\objref}[4]{\def\obj@rg{#4}%
  #1\ifx\obj@rg\empty#2\else#3\xspace\ref{#4}--\fi\ref}
\makeatother

\newcommand{\define}{\coloneqq}%

\newcounter{claims}

\let\savevec\vec
\usepackage{accents}
\let\vec\savevec

\usepackage[utf8]{inputenc}
\usepackage{pgfplots}
\DeclareUnicodeCharacter{2212}{−}
\usepgfplotslibrary{groupplots,dateplot}
\usetikzlibrary{patterns,shapes.arrows}
\usetikzlibrary{positioning,chains,fit,shapes,calc}
\pgfplotsset{compat=newest}

\usetikzlibrary{shapes.geometric,calc,arrows.meta,decorations.pathreplacing,positioning}

\usepackage{amsthm}
\newtheorem{theorem}{Theorem}[section]
\newtheorem{lemma}[theorem]{Lemma}
\newtheorem{proposition}[theorem]{Proposition}
\newtheorem{corollary}[theorem]{Corollary}
\newtheorem{conjecture}{Conjecture}[section]
\theoremstyle{definition}
\newtheorem{definition}[theorem]{Definition}

\theoremstyle{remark}
\newtheorem{remark}[theorem]{Remark}
\usepackage{etoolbox}
\AtBeginEnvironment{remark}{\itshape}

\crefname{lemma}{Lemma}{Lemmas}
\Crefname{lemma}{Lemma}{Lemmas}
\crefname{conjecture}{Conjecture}{Conjectures}
\Crefname{conjecture}{Conjecture}{Conjectures}

\usepackage[margin=1in]{geometry}

\usepackage{xcolor}

\usepackage[framemethod=TikZ]{mdframed}
\mdfdefinestyle{summarybox}{
  linewidth=0.5pt,
  roundcorner=0pt,
  innertopmargin=8pt,
  innerbottommargin=8pt,
  skipabove=12pt,
  skipbelow=12pt
}

\title{Optimal triangulations for piecewise linear approximations of non-convex variable products}

\author{
Robert Burlacu\thanks{Fraunhofer Institute for Integrated Circuits IIS, D-90411 N\"urnberg, Germany. \texttt{robert.burlacu@iis.fraunhofer.de}} 
\and 
Lukas Hager\thanks{Friedrich-Alexander-Universit\"at Erlangen-N\"urnberg, D-91058 Erlangen, Germany. \texttt{lukas.hager@fau.de}}
\and 
Robert Hildebrand\thanks{Grado Department of Industrial and Systems Engineering, Virginia Tech, Blacksburg, Virginia, USA. \texttt{rhil@vt.edu}. Partially funded by AFOSR grant FA9550-21-1-0107.}
}

\date{}

\begin{document}

\maketitle

\let\thefootnote\relax\footnotetext{We acknowledge the financial support by the Bavarian Ministry of Economic Affairs, Regional Development and Energy through the Center for Analytics -- Data -- Applications (ADA-Center) within the framework of ``BAYERN DIGITAL II''. Any opinions, findings, and conclusions or recommendations expressed in this material are those of the authors and do not necessarily reflect the views of the funding agencies.}

\begin{abstract}
\noindent We show optimal triangulations for piecewise linear (\pwl) approximations of indefinite quadratic functions over the plane.
Optimal triangulations have minimum triangle density while allowing a \pwl approximation that fulfills a prescribed error bound measured in the $L_\infty$-norm. 
In 2000, Pottmann et al. \cite{Pottmann:2000} proved optimal triangulations for \pwl interpolations and conjectured that these are also optimal for general \pwl approximations.
This conjecture was refuted in 2018 by Atariah et al. \cite{Atariah:2018}, who allowed a constant deviation at the vertices of the triangles and decreased the triangle density by roughly 3\%, though they left open whether their construction was optimal.
In this paper, we resolve this open question: allowing varying deviations at vertices reduces the optimal triangle density by 25\% compared to Atariah et al., and we prove this is globally optimal.
We thus show that the potential of general \pwl approximations is significantly lower for indefinite than for definite quadratic functions, where the triangle density can be halved when allowing general approximations compared to interpolations.\footnote{For definite quadratic functions $F(x,y) = x^2 + y^2$, Pottmann et al.\ \cite{Pottmann:2000} showed optimal interpolation has density $\propto 1/\varepsilon$, while Atariah et al.\ \cite{Atariah:2018} showed optimal general approximation has density $\propto 1/(2\varepsilon)$---a 50\% reduction. For indefinite functions, our improvement from interpolation ($1/(2\sqrt{5}\varepsilon)$) to general ($3\sqrt{3}/(32\varepsilon)$) is only $\approx 27\%$.}
Furthermore, we prove that among parallelogram tilings---triangulations built from translated copies of a triangle and its point-reflection---the constant-deviation construction of Atariah et al.\ is optimal when continuity of the \pwl approximation is required. We conjecture that this optimality extends to all continuous triangulations, not just those based on parallelograms.
\end{abstract}

\section{Introduction}
\label{sec:intro}

We consider piecewise linear (\pwl) approximations of the graph
of an indefinite quadratic function $ F\colon \R^2 \to \R,$
$$ F(x, y) = ax^2+2bxy+cy^2+dx+ey+g, $$
where $ac-b^2 < 0$, so the graph of $F$ is a saddle surface.
A \pwl approximation $ f \colon \R^2 \to \R $ of~$F$
is defined by an underlying triangulation
$ \triang = \{T_k\}_{k \in \N}$ of $\R^2$---where~$f$ is linear over each triangle~$ T_k $---together with a prescribed signed offset (deviation) from $F$ at each vertex of the triangulation.
We measure the approximation error between $F$ and $f$
in the $L_{\infty}$-norm, \ie the maximum pointwise deviation.
Our goal is to find triangulations that achieve a prescribed approximation error using the minimum number of triangles.
Optimal triangulations for definite quadratic functions were already established by Pottmann et al.\ \cite{Pottmann:2000} and Atariah et al.\ \cite{Atariah:2018}; this paper resolves the indefinite case.

\Cref{fig:pwl-types} illustrates the key concepts: the quadratic surface of the standard indefinite quadratic form $F(x,y) = xy$, a continuous \pwl approximation where the approximating surface is connected, and a general (discontinuous) \pwl approximation where the linear pieces may have gaps at shared boundaries. This paper focuses on general approximations, which provide greater flexibility and can achieve lower triangle density than continuous approximations.

\begin{figure}[ht]
\centering
\begin{tikzpicture}[scale=0.9]
\begin{scope}[shift={(-5,0)}]
\begin{axis}[
    view={60}{25},
    width=5cm, height=4.5cm,
    xlabel={$x$}, ylabel={$y$}, zlabel={$z$},
    xlabel style={font=\scriptsize},
    ylabel style={font=\scriptsize},
    zlabel style={font=\scriptsize},
    tick label style={font=\tiny},
    title={\footnotesize (a) Surface $F(x,y)=xy$},
    title style={yshift=-2pt},
    domain=-1.5:1.5,
    y domain=-1.5:1.5,
    samples=20,
    zmin=-2.5, zmax=2.5,
]
\addplot3[surf, shader=interp, colormap/cool, opacity=0.8] {x*y};
\end{axis}
\end{scope}

\begin{scope}[shift={(0,0)}]
\begin{axis}[
    view={60}{25},
    width=5cm, height=4.5cm,
    xlabel={$x$}, ylabel={$y$}, zlabel={$z$},
    xlabel style={font=\scriptsize},
    ylabel style={font=\scriptsize},
    zlabel style={font=\scriptsize},
    tick label style={font=\tiny},
    title={\footnotesize (b) Continuous PWL},
    title style={yshift=-2pt},
    zmin=-2.5, zmax=2.5,
]
\addplot3[patch, patch type=triangle, draw=blue!80, line width=0.8pt, fill=blue!10, opacity=0.5]
coordinates {
(-1.5,-1.5,-2.5) (-1.5,0,-2.5) (0,-1.5,-2.5)
(-1.5,0,-2.5) (0,0,-2.5) (0,-1.5,-2.5)
(-1.5,0,-2.5) (-1.5,1.5,-2.5) (0,0,-2.5)
(-1.5,1.5,-2.5) (0,1.5,-2.5) (0,0,-2.5)
(0,-1.5,-2.5) (0,0,-2.5) (1.5,-1.5,-2.5)
(0,0,-2.5) (1.5,0,-2.5) (1.5,-1.5,-2.5)
(0,0,-2.5) (0,1.5,-2.5) (1.5,0,-2.5)
(0,1.5,-2.5) (1.5,1.5,-2.5) (1.5,0,-2.5)
};
\addplot3[patch, patch type=triangle, shader=flat, draw=black!70, line width=0.6pt, fill=blue!30, opacity=0.7]
coordinates {
(-1.5,-1.5,-2.25) (-1.5,0,0) (0,-1.5,0)
(-1.5,0,0) (0,0,0) (0,-1.5,0)
(-1.5,0,0) (-1.5,1.5,-2.25) (0,0,0)
(-1.5,1.5,-2.25) (0,1.5,0) (0,0,0)
(0,-1.5,0) (0,0,0) (1.5,-1.5,-2.25)
(0,0,0) (1.5,0,0) (1.5,-1.5,-2.25)
(0,0,0) (0,1.5,0) (1.5,0,0)
(0,1.5,0) (1.5,1.5,2.25) (1.5,0,0)
};
\end{axis}
\end{scope}

\begin{scope}[shift={(5,0)}]
\begin{axis}[
    view={60}{25},
    width=5cm, height=4.5cm,
    xlabel={$x$}, ylabel={$y$}, zlabel={$z$},
    xlabel style={font=\scriptsize},
    ylabel style={font=\scriptsize},
    zlabel style={font=\scriptsize},
    tick label style={font=\tiny},
    title={\footnotesize (c) Discontinuous PWL},
    title style={yshift=-2pt},
    zmin=-2.5, zmax=2.5,
]
\addplot3[patch, patch type=triangle, draw=orange!80, line width=0.8pt, fill=orange!10, opacity=0.5]
coordinates {
(-1.5,-1.5,-2.5) (-1.5,0,-2.5) (0,-1.5,-2.5)
(-1.5,0,-2.5) (0,0,-2.5) (0,-1.5,-2.5)
(-1.5,0,-2.5) (-1.5,1.5,-2.5) (0,0,-2.5)
(-1.5,1.5,-2.5) (0,1.5,-2.5) (0,0,-2.5)
(0,-1.5,-2.5) (0,0,-2.5) (1.5,-1.5,-2.5)
(0,0,-2.5) (1.5,0,-2.5) (1.5,-1.5,-2.5)
(0,0,-2.5) (0,1.5,-2.5) (1.5,0,-2.5)
(0,1.5,-2.5) (1.5,1.5,-2.5) (1.5,0,-2.5)
};
\addplot3[patch, patch type=triangle, shader=flat, draw=black!70, line width=0.6pt, fill=orange!40, opacity=0.7]
coordinates {
(-1.5,-1.5,-1.75) (-1.5,0,0.5) (0,-1.5,0.5)
(-1.5,0,-0.5) (0,0,0.5) (0,-1.5,-0.5)
(-1.5,0,0.5) (-1.5,1.5,-1.75) (0,0,-0.5)
(-1.5,1.5,-2.75) (0,1.5,0.5) (0,0,0.5)
(0,-1.5,-0.5) (0,0,-0.5) (1.5,-1.5,-2.75)
(0,0,0.5) (1.5,0,0.5) (1.5,-1.5,-1.75)
(0,0,-0.5) (0,1.5,-0.5) (1.5,0,-0.5)
(0,1.5,0.5) (1.5,1.5,1.75) (1.5,0,0.5)
};
\end{axis}
\end{scope}
\end{tikzpicture}
\caption{Illustration of \pwl approximations. (a) The saddle surface $F(x,y) = xy$. (b) A continuous \pwl approximation: the linear pieces meet along shared edges, forming a connected surface. (c) A general (discontinuous) \pwl approximation: linear pieces may have different values at shared vertices, creating gaps. In (b) and (c), the triangulation is shown projected onto the $xy$-plane. General approximations allow larger triangles for the same error bound.}
\label{fig:pwl-types}
\end{figure}
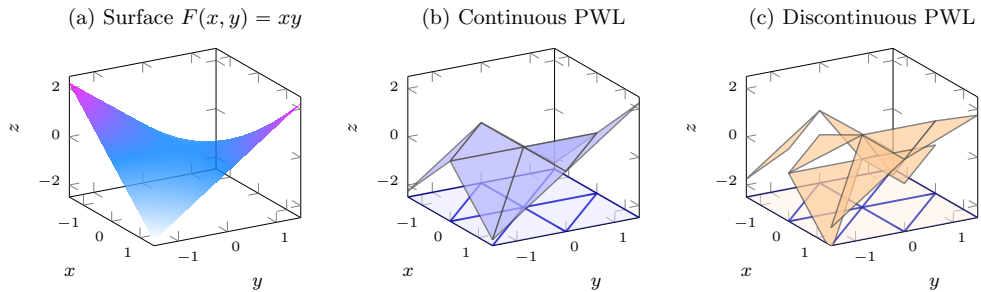

\paragraph{Scope: Unbounded Domains}
Throughout this paper, we study triangulations of the \emph{unbounded} domain $\R^2$. For compact domains, the optimal triangulation depends on the specific domain shape and boundary conditions, and may differ significantly from our results. A detailed study on \pwl approximations of indefinite quadratic forms over rectangular domains can be found in \cite{hager2023approximation}.
On $\R^2$, translation and reflection invariance (see \cref{sec:reduction}) allow us to reduce the problem to optimizing a single triangle that tiles the plane. This reduction does not apply to bounded domains, where boundary effects and domain-specific geometry may permit triangulations with lower density than those achievable through periodic tilings.
\paragraph{Triangle Density}

Since the domain $\R^2$ is unbounded, we cannot simply count the number of triangles to compare triangulations. Instead, we use the concept of \emph{triangle density} from \cite{Atariah:2018}, which measures the asymptotic number of triangles per unit area (see Definition~\ref{def:triangle-density} for the formal statement). Intuitively, triangle density counts how many triangles intersect a large square $Q_\ell$ of side length $\ell$, normalized by its area, as $\ell \to \infty$. \Cref{fig:triangle-density} illustrates this concept.

\begin{figure}[ht]
\centering
\makebox[\textwidth][c]{%
\begin{tikzpicture}[scale=0.85]
\foreach \i in {-3,...,4} {
    \foreach \j in {-3,...,4} {
        \draw[thin, blue!40] (\i,\j) -- (\i+1,\j) -- (\i,\j+1) -- cycle;
        \draw[thin, blue!40] (\i+1,\j) -- (\i+1,\j+1) -- (\i,\j+1) -- cycle;
    }
}
\draw[very thick, red!70!black, fill=red!10, fill opacity=0.3] (-1.5,-1.5) rectangle (2.5,2.5);
\node[red!70!black, font=\small] at (0.5, 2.8) {$Q_\ell$};
\foreach \i in {-2,...,2} {
    \foreach \j in {-2,...,2} {
        \fill[green!40, opacity=0.5] (\i,\j) -- (\i+1,\j) -- (\i,\j+1) -- cycle;
        \fill[green!40, opacity=0.5] (\i+1,\j) -- (\i+1,\j+1) -- (\i,\j+1) -- cycle;
    }
}
\draw[very thick, red!70!black] (-1.5,-1.5) rectangle (2.5,2.5);
\draw[<->, thick] (-1.5, -2.0) -- (2.5, -2.0);
\node[below, font=\small] at (0.5, -2.0) {$\ell$};
\draw[->, thick] (-3.5,0) -- (5.2,0) node[right] {$x$};
\draw[->, thick] (0,-3.5) -- (0,5.2) node[above] {$y$};
\node[align=left, font=\footnotesize, anchor=north west] at (5.5, 4.5) {
\textcolor{blue!60}{---} Triangulation $\TT$\\[2pt]
\textcolor{red!70!black}{\rule{8pt}{8pt}} Counting region $Q_\ell$\\[2pt]
\textcolor{green!50!black}{\rule{8pt}{8pt}} Triangles in count
};
\end{tikzpicture}%
}
\caption{Illustration of triangle density. The triangulation $\TT$ (blue) covers the plane. To compute the density, we count triangles intersecting the square $Q_\ell$ (highlighted in green) and divide by $\ell^2$. As $\ell \to \infty$, this ratio converges to the triangle density $\delta(\TT)$.}
\label{fig:triangle-density}
\end{figure}
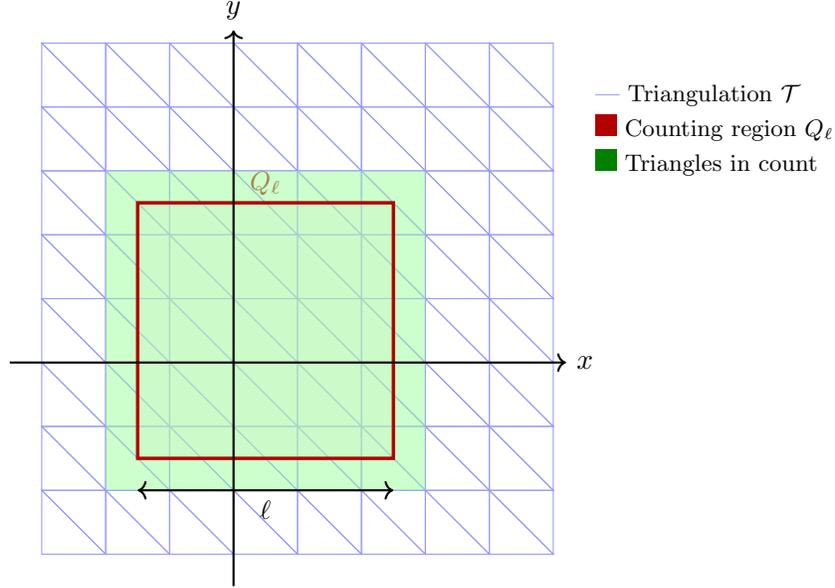

For a triangulation where all triangles have equal area $A$, the triangle density is approximately $1/A$ (exactly $1/A$ for tilings). Thus, \emph{larger triangles correspond to lower density}, and minimizing the triangle density is equivalent to maximizing the triangle area if all triangles have the same area.

\paragraph{Optimal Triangle Density}

Given an error bound $\varepsilon > 0$, we call a triangulation $\TT$ an \emph{$\varepsilon$-triangulation} if there exists a \pwl function $f$ based on $\TT$ such that $\|f - F\|_\infty \leq \varepsilon$. The \emph{optimal triangle density} $\delta^*_{F,\varepsilon}$ (or $\delta^*_\varepsilon$ for short, when the function $F$ is clear from context) is the infimum of triangle densities over all $\varepsilon$-triangulations (see Definition~\ref{def:optimal-triangle-density} for the formal statement, where $G$ denotes a general function).
A triangulation achieving this infimum is called \emph{$\varepsilon$-optimal}. Since the optimal density scales as $1/\varepsilon$, we derive triangulations parameterized by $\varepsilon$ and refer to them simply as \emph{optimal triangulations}.

\paragraph{Goal of This Paper}
Our primary goal is to determine the optimal triangle density for \pwl approximations of indefinite quadratic functions, and to characterize the triangulations that achieve it. We address this for several classes of approximations:
\begin{itemize}[nosep]
\item \textbf{General approximations}: The \pwl function $f$ may deviate from $F$ at vertices, and need not be continuous.
\item \textbf{Continuous approximations}: The function $f$ must be continuous (equivalent to requiring equal deviations at shared vertices).
\item \textbf{Interpolations}: The function $f$ must satisfy $f(v) = F(v)$ at all vertices $v$.
\item \textbf{Over/underestimations}: The function $f$ must satisfy $f \geq F$ (overestimation) or $f \leq F$ (underestimation) everywhere.
\end{itemize}
For each class, we seek to find the optimal triangle density and prove that it is indeed optimal---not merely an upper bound.

\paragraph{Reduction to a Single Triangle}
A key simplification allows us to reduce the problem of finding an optimal triangulation of the entire plane to finding a single optimal triangle. This reduction relies on two invariance properties of quadratic functions: invariance under translation (Lemma~\ref{lem:invariance_trans}) and invariance under point reflection (Lemma~\ref{lem:invariance_reflection}). While these properties were known in $\R^2$ for specific cases \cite{Atariah:2018,Pottmann:2000}, we state them for general quadratic forms in $\R^n$. For $n=2$, these invariances yield a reduction to a single triangle (Lemma~\ref{lem:area_density}) because a triangle and its point-reflection, which has the same approximation accuracy, form a parallelogram that tiles the plane. 

These properties imply that if a triangle $T$ admits a linear $\varepsilon$-approximation, then the plane can be tiled by translated copies of $T$ and its reflection $T'$, with each copy admitting an $\varepsilon$-approximation. The resulting triangulation has density exactly $1/A(T)$, where $A(T)$ is the area of $T$ (see Lemma~\ref{lem:area_density}).

Consequently, \emph{minimizing triangle density with respect to $\varepsilon$  is equivalent to maximizing the area of a single triangle that admits an $\varepsilon$-approximation}. This transforms the infinite-dimensional problem of optimizing over all triangulations into a finite-dimensional optimization problem over triangle vertices and deviations. We formalize and solve this optimization problem in \cref{sec:solution}.

\subsection{Related Work}
\label{sec:related-work}

The problem of optimal triangulations for function approximation has a rich history in computational geometry; see \cite{BernEppstein:1995,Toth:2017} for comprehensive surveys.
The principle that optimal triangles for piecewise linear interpolation should be aligned with Hessian eigenvectors was established by D'Azevedo and Simpson \cite{DazevadoSimpson:1991} and extended by D'Azevedo \cite{Dazevado:1991}; our results extend this to the indefinite case with non-interpolatory approximations.
\paragraph{Definite case}
The specific problem of \pwl approximations of quadratic functions over the plane was first addressed by Pottmann et al.\ \cite{Pottmann:2000} in 2000. 
They completely solved the \emph{definite} case ($ac - b^2 > 0$), where the graph of $F$ is an elliptic paraboloid. In this setting, the quadratic form induces a Euclidean metric on $\R^2$, and the approximation error over a triangle is controlled by its circumradius: the maximum error of the best linear approximation equals the squared circumradius (for interpolation) or half the squared circumradius (for general approximation). The optimal triangulations are therefore regular (equilateral) triangles---the classical solution to the lattice covering problem---and allowing vertex deviations exactly halves the triangle density compared to interpolation \cite{Atariah:2018}. The definite case thus reduces to a well-understood problem in discrete geometry.

\paragraph{Indefinite case}
The \emph{indefinite} case ($ac - b^2 < 0$), where the graph is a saddle surface, is fundamentally different. The quadratic form induces a \emph{pseudo-Euclidean} (indefinite) metric on $\R^2$, under which ``distances'' can be zero along asymptotic directions. Unlike the definite case, the approximation error depends not just on the size of the triangle but critically on its \emph{orientation}: the error along an edge is governed by the edge product $k = (x_2 - x_1)(y_2 - y_1)$ rather than the edge length, and axis-parallel edges contribute zero curvature since $F(x,y) = xy$ is linear along them. This orientation-dependence creates a genuinely nonconvex optimization landscape in which the interplay between triangle geometry and vertex deviations is much richer.

For this indefinite case, Pottmann et al.\ showed optimal \pwl approximations only under the restriction to interpolations, where the approximation must coincide with $F$ at the vertices of the triangulation.
They conjectured that optimal interpolations were also optimal \pwl approximations---in other words, that deviating at vertices would not lead to better approximations with fewer triangles.
Atariah et al.\ (2018) refuted this conjecture in their study \cite{Atariah:2018}, demonstrating that allowing a constant deviation $\Delta = \tfrac{\varepsilon}{3}$ at each vertex improves upon interpolation. This adjustment allowed for a modest reduction in triangle density, approximately 3\%, when compared to the optimal interpolations previously established. Specifically, the triangle density was reduced from $\tfrac{1}{2\sqrt{5}\varepsilon}$ ($\approx 0.224$, with $\varepsilon=1$) as reported by Pottmann et al.\ (2000) \cite{Pottmann:2000}, to $\tfrac{\sqrt{3}}{8\varepsilon}$ ($\approx 0.2165$, with $\varepsilon=1$)  for the indefinite standard form $F(x,y) = xy$. However, Atariah et al.\ explicitly noted that they did not know whether their constant-deviation construction was optimal, leaving this as an open question. The present paper resolves this question: varying deviations at vertices yield a further 25\% reduction in triangle density, and we prove this is globally optimal.

From an optimization perspective, indefinite quadratic functions such as $F(x,y) = xy$ arise naturally as bilinear terms in mixed-integer nonlinear programming (MINLP). 
The foundational work of McCormick \cite{McCormick:1976} introduced convex envelope relaxations for such terms, and subsequent research \cite{Boland:2017,hager:2022} has studied the quality of these relaxations. 
Piecewise linear approximations provide an alternative approach to handling bilinear terms, which motivates the study of optimal triangulations from an optimization perspective.
Related work on surface approximation using quadric error metrics \cite{GarlandHeckbert:1997,GarlandHeckbert:1999} and optimal Delaunay triangulations \cite{Chen:2004} addresses similar geometric optimization problems, though with different objective functions and constraints.

In practice, regular (axis-aligned) grids are commonly used for \pwl approximations due to their simplicity \cite{hager:2022}. However, regular grids are significantly suboptimal for bilinear functions. Since $F(x,y) = xy$ is linear along axis-parallel directions, axis-aligned edges provide no curvature information, and the approximation error depends solely on vertex deviations. The maximal admissible triangle areas of $4\varepsilon$ (non-continuous) and $2\varepsilon$ (continuous) for regular grids \cite{hager2023approximation} correspond to triangle densities of $\tfrac{1}{4\varepsilon}$ and $\tfrac{1}{2\varepsilon}$, respectively. In contrast, our optimally oriented triangulations achieve areas of $\tfrac{32\sqrt{3}}{9}\varepsilon$ (general) and $\tfrac{8\sqrt{3}}{3}\varepsilon$ (continuous)---improvements of roughly 54\% and 131\% over regular grids, respectively (see Table~\ref{tab:intro-summary} for all cases).
\paragraph{Contributions}

In this paper, we develop a unified framework for optimal \pwl approximations of indefinite quadratic functions that recovers known results and establishes new ones. Table~\ref{tab:intro-summary} summarizes all optimal triangulations.

\begin{table}[ht]
\centering
\begin{tabular}{|l|c|c|c|c|l|}
\hline
Approximation Type 
& Triangle Density 
& $\varepsilon=1$ 
& Optimal Area 
& $\varepsilon=1$ 
& Status \\
\hline
\textbf{General} 
& $\tfrac{3\sqrt{3}}{32\varepsilon}$ 
& 0.16 
& $\tfrac{32\sqrt{3}}{9}\varepsilon$ 
& 6.16 
& \textbf{New} \\
\hline
Continuous
& $\tfrac{\sqrt{3}}{8\varepsilon}$ 
& 0.22 
& $\tfrac{8\sqrt{3}}{3}\varepsilon$ 
& 4.62 
& constr.\ old\\
(parallelogram) 
& & & & & opt.\ new \\
\hline
Interpolation 
& $\tfrac{1}{2\sqrt{5}\varepsilon}$ 
& 0.22 
& $2\sqrt{5}\varepsilon$ 
& 4.47 
& Recovered \\
\hline
\textbf{Overestimation} 
& $\tfrac{3\sqrt{3}}{16\varepsilon}$ 
& 0.32 
& $\tfrac{16\sqrt{3}}{9}\varepsilon$ 
& 3.08 
& \textbf{New} \\
\hline
\textbf{Underestimation} 
& $\tfrac{3\sqrt{3}}{16\varepsilon}$ 
& 0.32 
& $\tfrac{16\sqrt{3}}{9}\varepsilon$ 
& 3.08 
& \textbf{New} \\
\hline
\textbf{Cont.\ Over/Under}
& $\tfrac{\sqrt{3}}{4\varepsilon}$ 
& 0.43 
& $\tfrac{4\sqrt{3}}{3}\varepsilon$ 
& 2.31 
& \textbf{New} \\(parallelogram)
& 
&  
& 
&
&\\
\hline
\end{tabular}

\caption{Summary of optimal triangulations for $F(x,y) = xy$ with error bound $\varepsilon$. Results marked ``New'' are first established in this paper. ``Constr.\ known; opt.\ new'' means the construction was given by Atariah et al.\ \cite{Atariah:2018} but optimality among parallelogram tilings is first proved here (see \cref{conj:continuous-optimality} for the extension to all continuous triangulations). ``Recovered'' means the result was established by Pottmann et al.\ \cite{Pottmann:2000} and our framework recovers it as a special case.}
\label{tab:intro-summary}
\end{table}

We now state our main results formally. By Lemma~\ref{lem:reduction-standard}, all results for the standard form $F(x,y) = xy$ extend to general indefinite quadratic functions via a scaling factor.

\begin{theorem}[General \pwl approximations]
\label{thm:optimal-triangulation}
For any approximation accuracy $\varepsilon > 0$, the optimal triangle density for \pwl $\varepsilon$-approximations of $F(x,y) = xy$ is $\tfrac{3\sqrt{3}}{32\varepsilon}$. The optimal triangles have area $\tfrac{32\sqrt{3}}{9}\varepsilon$.
\end{theorem}
\noindent\textit{Proof in Section~\ref{sec:proofs}, page~\pageref{thm:optimal-triangulation-proof}.} The optimal triangle is illustrated in \cref{fig:optimal-triangle}.
\begin{figure}[ht]
\centering
\begin{tikzpicture}[scale=0.9]
    \definecolor{col0}{RGB}{220,60,60}    %
    \definecolor{col1}{RGB}{60,160,60}    %
    
    \coordinate (v1) at (0,0);
    \coordinate (v2) at (3.6,0.97);
    \coordinate (v3) at (0.97,3.6);
    
    \fill[blue!10] (v1) -- (v2) -- (v3) -- cycle;
    
    \draw[thick] (v1) -- (v2) node[midway, below right] {$k_{12} = \tfrac{32\varepsilon}{9}$};
    \draw[thick] (v1) -- (v3) node[midway, above left] {$k_{13} = \tfrac{32\varepsilon}{9}$};
    \draw[thick] (v2) -- (v3) node[midway, right] {$k_{23} = -\tfrac{64\varepsilon}{9}$};
    
    \fill[col0] (v1) circle (3pt) node[below left] {$v_1$: $d_1 = -\varepsilon$};
    \fill[col1] (v2) circle (3pt) node[right] {$v_2$: $d_2 = \tfrac{7\varepsilon}{9}$};
    \fill[col1] (v3) circle (3pt) node[above] {$v_3$: $d_3 = \tfrac{7\varepsilon}{9}$};
    
    \draw[dashed, gray] (0,0) -- (2.3,2.3);
    \node[gray] at (1.6,1) {\tiny $x=y$};
    
    \node at (1.5,1.5) {$A^* = \tfrac{32\sqrt{3}}{9}\varepsilon$};
\end{tikzpicture}
\caption{The optimal triangle for $\varepsilon$-approximation of $F(x,y) = xy$. Each vertex $v_i$ has a \emph{deviation} $d_i := f(v_i) - F(v_i)$, the signed difference between the approximating plane and the surface at that vertex. Each edge carries an \emph{edge product} $k_{ij} := (x_j - x_i)(y_j - y_i)$, which measures the curvature of $F$ along that edge (positive for ``ascending'' edges, negative for the ``descending'' edge). Note the symmetry: $v_2$ and $v_3$ are reflections across the line $x=y$, and $d_2 = d_3$. We call this the ``LHH'' (Low-High-High) deviation pattern: one low deviation ($d_1 = -\varepsilon$) and two equal high deviations ($d_2 = d_3 = 7\varepsilon/9$). These quantities are defined formally in Section~\ref{sec:error-analysis}.}
\label{fig:optimal-triangle}
\end{figure}
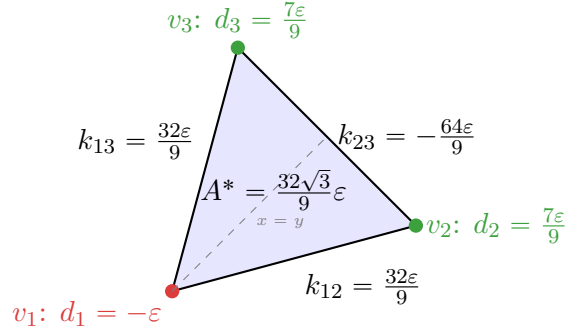

\begin{theorem}[\Pwl underestimations and overestimations]
\label{thm:optimal-triangulation-under}
Under the constraint of underestimation ($f \leq F$) or overestimation ($f \geq F$), the optimal triangle density is $\tfrac{3\sqrt{3}}{16\varepsilon}$. The optimal triangles have area $\tfrac{16\sqrt{3}}{9}\varepsilon$.
\end{theorem}
\noindent\textit{Proof in Section~\ref{sec:proofs}, page~\pageref{thm:optimal-triangulation-under-proof}.}

\begin{theorem}[Continuous \pwl approximations -- parallelogram tilings]
\label{thm:optimal-triangulation-cont}
Among parallelogram tilings, the optimal triangle density for continuous \pwl $\varepsilon$-approximations of $F(x,y) = xy$ is $\tfrac{\sqrt{3}}{8\varepsilon}$. The optimal triangles have area $\tfrac{8\sqrt{3}}{3}\varepsilon$ and constant deviation $\Delta = \tfrac{\varepsilon}{3}$ at all vertices. We conjecture this is optimal among all continuous triangulations (\cref{conj:continuous-optimality}).
\end{theorem}
\noindent\textit{Proof in Section~\ref{sec:proofs}, page~\pageref{thm:optimal-triangulation-cont-proof}.}

\begin{theorem}[Continuous \pwl under/overestimations -- parallelogram tilings]
\label{thm:optimal-triangulation-cont-under}
Among parallelogram tilings, the optimal triangle density for continuous \pwl $\varepsilon$-overestimations (resp.\ $\varepsilon$-underestimations) of $F(x,y)=xy$ is $\tfrac{\sqrt{3}}{4\varepsilon}$. The optimal triangles have area $\tfrac{4\sqrt{3}}{3}\varepsilon$ and constant deviation $\Delta = \tfrac{2\varepsilon}{3}$ (for overestimation) or $\Delta = -\tfrac{2\varepsilon}{3}$ (for underestimation).
\end{theorem}
\noindent\textit{Proof.} This follows from the continuous case analysis (Theorem~\ref{thm:optimal-triangulation-cont}) applied to restricted deviation ranges. For overestimation ($f \geq F$), the vertex deviations are constrained to $[0, \varepsilon]$; for underestimation ($f \leq F$), to $[-\varepsilon, 0]$. In both cases, the feasible error interval has width $\varepsilon$ instead of $2\varepsilon$. Continuity forces a constant deviation $\Delta$ at all vertices, and the same optimization yields $\Delta^* = 2\varepsilon/3$ (overestimation) or $\Delta^* = -2\varepsilon/3$ (underestimation), with optimal area $A^* = \tfrac{4\sqrt{3}}{3}\varepsilon$ and density $\tfrac{\sqrt{3}}{4\varepsilon}$. \qed

The key insight enabling these results is a reformulation via slack variables and edge-product bounds that reduces the nonconvex optimization to a structured problem over a box, which is then solved through separate convexity, symmetry arguments, and an algebraic optimality certificate.

\paragraph{Outline}
The paper is organized as follows. Section~\ref{sec:opt} introduces the formal definitions of triangulations, \pwl approximations, and triangle density. Section~\ref{sec:reduction} establishes that finding optimal triangulations of the plane reduces to finding a single optimal triangle, using invariance properties of quadratic functions and the fact that a triangle and its point-reflection tile the plane. Section~\ref{sec:error-analysis} analyzes the approximation error over a single triangle and introduces pseudo-Euclidean motion for geometric interpretation. Section~\ref{sec:solution} formulates and solves the optimization problem for finding an optimal triangle: the model is first formulated and then reformulated via slack variables and edge-product bounds (Section~\ref{sec:slack-reformulation}), yielding a structured optimization over box-constrained edge products that is solved through separate convexity, symmetry arguments, and algebraic optimality certificates (Section~\ref{sec:elem-proof}). Section~\ref{sec:proofs} presents the proofs of our main theorems and discusses special cases.

\section{Definitions}
\label{sec:opt}

This section establishes the formal framework for our analysis. We define triangulations, \pwl approximations, approximation error, and triangle density---the key metric for comparing triangulations of unbounded domains.
\begin{definition}[Triangle]
	\label{def:simplex}
	A triangle $T$ is the convex hull
	of three affinely independent points in $ \R ^d $ with $d\in \N$.
\end{definition}
A triangulation is a partition consisting of triangles.

\begin{definition}[Triangulation]
	\label{def:triangulation_fully}
	A collection of triangles $ \TT = \{T_i\}_{i\in I} $
	is called a \emph{triangulation} of a set $ X \subseteq \R^2 $
	if $ X = \bigcup_{i \in I} T_i $
	and for any distinct $i, j \in I$, the intersection $ T_i \cap T_j $
	is either empty or a common proper face of both $T_i$ and $T_j$. 
	If $X = \R^2$, we additionally assume $\TT$ is locally finite (see Definition~\ref{def:pwl-functions}).
\end{definition}
Furthermore, we introduce notation for the minimal index triangle in a triangulation, which will be used later in the context of non-continuous approximations. Let $\TT = \{T_i\}_{i\in I}$ be a triangulation of a set $X \subseteq \R^2$. For each $x \in X$, we define the minimum index as $i_x := \min\{i \in \mathcal{I} : x \in T_i\}$.
We use triangulations to define \pwl functions.
Let $B_\infty(r) = \{x \in \R^2 : \|x\|_\infty \leq r\}$ be a ball centered at the origin of $\R^2$.

\begin{definition}[Locally Finite Triangulation and Piecewise Linear Functions]
	\label{def:pwl-functions}
	Given $ X \subseteq \R^2 $,
	let $ \TT \define \{T_i\}_{i \in \mathcal I}$
	be a collection of triangles with countable index set $\mathcal I$.
	We call $\mathcal T$ a \emph{locally finite triangulation} of~$X$ if 
	$ X = \bigcup_{i \in \mathcal I} T_i $,
	the intersection $ T_i \cap T_j $
	of any two triangles $ T_i, T_j \in \triang $
	is a proper face of $T_i$ and $T_j$, and
	for all $R > 0$, the ball $B_\infty(R)$ intersects only a finite number of triangles $T_i$.
	
	A function~$ g\colon X \rightarrow \R $
	is called \emph{piecewise linear} with respect to $\TT$
	if there exist vectors $ \mathbf{m}_i \in \R^2 $
	and constants $ c_i \in \R $ for $ i \in \mathcal I$
	such that
	\begin{equation}
		\label{eq:pwl-function}
		g (\mathbf{x}) = \mathbf{m}_{i_{\mathbf{x}}}^\top \mathbf{x} + c_{i_\mathbf{x}}
		\quad \text{if } \mathbf{x} \in T_{i_{\mathbf{x}}},
	\end{equation}
	where $i_\mathbf{x} = \min\{i \in \mathcal{I} : \mathbf{x} \in T_i\}$ is the minimum index among all triangles containing $\mathbf{x}$.
We denote the set of \pwl functions from $\R^2$ to $\R$ as $\mathcal{PWL}$.
\end{definition}

\begin{definition}[Approximation Error]
	\label{def:linearization-error}
	Consider a triangulation $ \TT $ of $ X \subseteq \R^2 $
	and let $ g \colon X \to \R $ be a \pwl approximation
	of a function $ G\colon X \to \R $  \wrt $ \TT $.
	We call
	\begin{equation}
		E_{g, G}\colon X \to \R, \quad E_{g, G}(\mathbf{x}) \define g(\mathbf{x}) - G(\mathbf{x})
	\end{equation}
	the \emph{error function} \wrt $g$ and $G$ and
	\begin{equation}
		\varepsilon_{g, G}(T) \define \max_{\mathbf{x} \in T} \abs{E_{g, G}(\mathbf{x})}
	\end{equation}
	the \emph{approximation error} on a triangle~$ T \in \triang $.
	Consequently, we define the \emph{ approximation error} of $g$
	(or, equivalently, of $ \triang $) \wrt $G$ over the domain~$X$ as
	\begin{equation}
		\varepsilon_{g, G}(\triang)
		\define \max_{T \in \triang} \varepsilon_{g, G}(T).
	\end{equation}
	Given some $ \varepsilon > 0 $, we call $g$ an $ \varepsilon $-approximation
	and $ \TT $
	an $ \varepsilon $-triangulation
	if the approximation error is smaller than or equal to~$ \varepsilon $.
\end{definition}

For bounded domains, comparing triangulations is straightforward: we simply count the number of triangles. For unbounded domains like $\R^2$, this approach fails since the number is infinite. Instead, we use the concept of \emph{triangle density} from \cite{Atariah:2018}, which measures the asymptotic number of triangles per unit area.

\begin{definition}[Triangle Density \cite{Atariah:2018}]
\label{def:triangle-density}

Let $\TT$ be a triangulation of $\R^2$. 
The \emph{triangle density}  counts the
number of triangles of a triangulation that intersect the square $Q_\ell$ for larger
and larger side length $\ell$, in comparison to the area of $Q_\ell$,
$$ \delta(\mathcal T) := \limsup_{\ell \to \infty} \frac{\abs{\{T \in \TT: T \cap Q_\ell \neq \emptyset \}}}{\ell^2}.$$
For simplicity, we say a \pwl approximation has a \emph{density} of $\delta$ if the underlying triangulation $\TT$ has a triangle density of $\delta$.
\end{definition}

\begin{remark}
The $\limsup$ in Definition~\ref{def:triangle-density} always exists (possibly $+\infty$). For the periodic tilings constructed in this paper---parallelogram lattices generated by a triangle $T$ and its reflection $T'$---the limit exists and equals exactly $1/A(T)$, where $A(T)$ is the area of a single triangle.
\end{remark}

With this metric, we can define optimality for triangulations of the plane.

\begin{definition}[Optimal Triangle Density]
\label{def:optimal-triangle-density}
Given a function $G:\R^2\to \R$ and an approximation error bound $\varepsilon>0$, we define the $\varepsilon$-optimal triangle density of $G$ as 
    \begin{equation}
\delta^*_{G,\varepsilon} = \inf_{g \in \mathcal{PWL}}\{\delta(\mathcal T_g) :  \|g - G\|_\infty \leq \varepsilon\}.
\end{equation}
If there exists a triangulation $\TT_g$ allowing for a \pwl function $g$ such that $\delta(\TT_g)=\delta^*_{G,\varepsilon}$, we call $g$ an $\varepsilon$-optimal approximation of $G$ and $\TT_g$ an $\varepsilon$-optimal triangulation with respect to $G$. 
\end{definition}

\section{Reduction to standard indefinite form and a single triangle}
\label{sec:reduction}
The key insight enabling our results is that optimizing over all triangulations of the plane reduces to optimizing a single triangle. This reduction relies on two invariance properties of quadratic functions---translation and reflection---which we establish first for general quadratic forms in $\R^n$. We then specialize to $\R^2$, reduce to the indefinite standard form $F(x,y) = xy$, and show how a single optimal triangle generates an optimal triangulation of the entire plane.

\paragraph{Notation} The invariance lemmas (Lemmas~\ref{lem:invariance_trans} and~\ref{lem:invariance_reflection}) are stated for general $\R^n$ using boldface for vectors: $\mathbf{x} \in \R^n$, $\mathbf{c} \in \R^n$, etc. From Section~\ref{sec:standard-form} onward, we specialize to $\R^2$ and write points as $(x,y)$ without boldface, which is the convention for the remainder of the paper.
\subsection{Invariance under translation}
Quadratic surfaces ``look the same'' at every point: translating a simplex preserves the quality of any linear approximation.

\begin{lemma}[Invariance under translation]
\label{lem:invariance_trans}
Let $f(\mathbf{x}) = \mathbf{x}^\top Q \mathbf{x} + \mathbf{c}^\top \mathbf{x} + \alpha$ be a quadratic form with $Q \in \mathbb{R}^{n \times n}$ symmetric. For any simplex $S \subset \mathbb{R}^n$ and translation vector $\mathbf{t} \in \mathbb{R}^n$: if $\ell$ is a linear $\varepsilon$-approximation of $f$ over $S$, then there exists a linear $\varepsilon$-approximation $\ell'$ of $f$ over the translated simplex $S + \mathbf{t}$.
\end{lemma}
\begin{proof}
Let $\ell(\mathbf{x}) = \mathbf{d}^\top \mathbf{x} + \beta$. We seek $\ell'(\mathbf{x}) = {\mathbf{d}'}^\top \mathbf{x} + \beta'$ such that
\[
f(\mathbf{x} + \mathbf{t}) - \ell'(\mathbf{x} + \mathbf{t}) = f(\mathbf{x}) - \ell(\mathbf{x}) \quad \text{for all } \mathbf{x}.
\]
Expanding $f(\mathbf{x} + \mathbf{t}) = f(\mathbf{x}) + 2\mathbf{x}^\top Q \mathbf{t} + \mathbf{t}^\top Q \mathbf{t} + \mathbf{c}^\top \mathbf{t}$ and $\ell'(\mathbf{x} + \mathbf{t}) = {\mathbf{d}'}^\top \mathbf{x} + {\mathbf{d}'}^\top \mathbf{t} + \beta'$, the condition becomes:
\[
\ell'(\mathbf{x} + \mathbf{t}) - \ell(\mathbf{x}) = 2\mathbf{x}^\top Q \mathbf{t} + \mathbf{t}^\top Q \mathbf{t} + \mathbf{c}^\top \mathbf{t}.
\]
Matching coefficients of $\mathbf{x}$ gives $\mathbf{d}' = \mathbf{d} + 2Q\mathbf{t}$. Matching constants gives $\beta' = \beta - \mathbf{t}^\top Q \mathbf{t} + \mathbf{c}^\top \mathbf{t} - \mathbf{d}^\top \mathbf{t}$.

With these choices, $|f(\mathbf{y}) - \ell'(\mathbf{y})| = |f(\mathbf{y} - \mathbf{t}) - \ell(\mathbf{y} - \mathbf{t})| \leq \varepsilon$ for all $\mathbf{y} \in S + \mathbf{t}$.
\end{proof}

\subsection{Invariance under reflection}
Similarly, reflecting a simplex through the origin preserves the approximation quality.

\begin{lemma}[Invariance under reflection]
\label{lem:invariance_reflection}
Let $f(\mathbf{x}) = \mathbf{x}^\top Q \mathbf{x} + \mathbf{c}^\top \mathbf{x} + \alpha$ be a quadratic form with $Q \in \mathbb{R}^{n \times n}$ symmetric. For any simplex $S \subset \mathbb{R}^n$: if $\ell$ is a linear $\varepsilon$-approximation of $f$ over $S$, then there exists a linear $\varepsilon$-approximation $\ell'$ of $f$ over the reflected simplex $-S := \{-\mathbf{x} : \mathbf{x} \in S\}$.
\end{lemma}
\begin{proof}
Let $\ell(\mathbf{x}) = \mathbf{d}^\top \mathbf{x} + \beta$. We seek $\ell'(\mathbf{x}) = {\mathbf{d}'}^\top \mathbf{x} + \beta'$ such that $f(-\mathbf{x}) - \ell'(-\mathbf{x}) = f(\mathbf{x}) - \ell(\mathbf{x})$ for all $\mathbf{x}$.

Since $f(-\mathbf{x}) = \mathbf{x}^\top Q \mathbf{x} - \mathbf{c}^\top \mathbf{x} + \alpha$ and $\ell'(-\mathbf{x}) = -{\mathbf{d}'}^\top \mathbf{x} + \beta'$, we need:
\[
\mathbf{x}^\top Q \mathbf{x} - \mathbf{c}^\top \mathbf{x} + \alpha + {\mathbf{d}'}^\top \mathbf{x} - \beta' = \mathbf{x}^\top Q \mathbf{x} + \mathbf{c}^\top \mathbf{x} + \alpha - \mathbf{d}^\top \mathbf{x} - \beta.
\]
Matching coefficients gives $\mathbf{d}' = 2\mathbf{c} - \mathbf{d}$ and $\beta' = \beta$.

With these choices, $|f(\mathbf{y}) - \ell'(\mathbf{y})| = |f(-\mathbf{y}) - \ell(-\mathbf{y})| \leq \varepsilon$ for all $\mathbf{y} \in -S$.
\end{proof}

Note that the invariance properties above hold for general quadratic forms in $\R^n$, not just $\R^2$.

\subsection{Reduction to standard form}
\label{sec:standard-form}
We now specialize to $\R^2$. The following lemma shows that for indefinite quadratic forms, it suffices to study the standard form $F(x,y) = xy$.

\begin{lemma}[Reduction to standard form]
\label{lem:reduction-standard}
Let $G(x,y) = ax^2 + 2bxy + cy^2 + dx + ey + g$ be an indefinite quadratic form (i.e., $D := ac - b^2 < 0$). Then there exists a linear transformation $\Phi: (x,y) \mapsto (u,v)$ and a nonzero constant $\kappa$ such that the quadratic part of $G$ becomes $\kappa \cdot uv$. Consequently, any triangulation for $F(u,v) = uv$ can be transported to a triangulation for $G$ via the inverse linear map, with the approximation tolerance rescaled by $|\kappa|$. More precisely, if $f$ is a \pwl $\varepsilon$-approximation of $F(u,v)=uv$ on a triangulation $\mathcal{T}$ in $(u,v)$-coordinates, then $(f\circ \Phi)$ is a \pwl $(|\kappa|\varepsilon)$-approximation of $G$ on the transported triangulation $\Phi^{-1}(\mathcal{T})$ in $(x,y)$-coordinates. Equivalently, an $\varepsilon$-approximation of $G$ corresponds to an $(\varepsilon/|\kappa|)$-approximation of $uv$.
\end{lemma}
\begin{proof}
We consider three cases based on the coefficients. In each case, linear terms can be handled separately by adding a linear function to the approximation.

\textbf{Case 1:} $a = c = 0$. Then $D = -b^2 < 0$ requires $b \neq 0$, and $G = 2bxy + \text{(linear)}$. The substitution $(u,v) = (x, 2by)$ gives $G = uv + \text{(linear)}$, with area scaling factor $|2b|$.

\textbf{Case 2a:} $c = 0$, $a \neq 0$. Then $G = ax^2 + 2bxy + \text{(linear)}$ with $b \neq 0$ (since $D = -b^2 < 0$). The shear transformation 
\[
x = u, \quad y = v - \frac{a}{2b}u
\]
gives $ax^2 + 2bxy = 2b \cdot uv$, with area scaling factor $|2b|$.

\textbf{Case 2b:} $a = 0$, $c \neq 0$. Similarly, $G = cy^2 + 2bxy + \text{(linear)}$ with $b \neq 0$. The shear transformation
\[
y = v, \quad x = u - \frac{c}{2b}v
\]
gives $cy^2 + 2bxy = 2b \cdot uv$, with area scaling factor $|2b|$.

\textbf{Case 3:} $a \neq 0,$ $c \neq 0$. The asymptotic directions of $ax^2 + 2bxy + cy^2 = 0$ have slopes $m_\pm = \frac{-b \pm \sqrt{b^2-ac}}{c}$, which are the roots of $cm^2 + 2bm + a = 0$. The transformation
\begin{equation}
\label{eq:reduction-transform}
\begin{pmatrix} x \\ y \end{pmatrix} = \begin{pmatrix} 1 & 1 \\ m_+ & m_- \end{pmatrix} \begin{pmatrix} u \\ v \end{pmatrix}
\end{equation}
maps these asymptotes to the coordinate axes. Substituting into the quadratic part:
\begin{align*}
ax^2 + 2bxy + cy^2 &= (a + 2bm_+ + cm_+^2)u^2 + (a + 2bm_- + cm_-^2)v^2 \\
&\quad + 2(a + b(m_+ + m_-) + cm_+ m_-)uv.
\end{align*}
Since $m_\pm$ satisfy $cm^2 + 2bm + a = 0$, the coefficients of $u^2$ and $v^2$ vanish. By Vieta's formulas, $m_+ + m_- = -\tfrac{2b}{c}$ and $m_+ m_- = \tfrac{a}{c}$, so the coefficient of $uv$ is $\tfrac{4D}{c}$. The Jacobian is $|m_+ - m_-| = \tfrac{2\sqrt{|D|}}{|c|}$.

In all cases, the quadratic part becomes $\kappa \cdot uv$ for some $\kappa = \tfrac{4D}{c} = \tfrac{4(ac - b^2)}{c} \neq 0$. Absorbing $|\kappa|$ into one variable gives the standard form $F(u,v) = uv$. In this normalization, triangle density scales with the Jacobian of $\Phi$, while the approximation tolerance scales by $|\kappa|$.

\textbf{Error preservation:} Let $\Phi: (x,y) \mapsto (u,v)$ be the linear transformation with Jacobian $J = |\det(D\Phi)|$. If $\ell$ is a linear $\varepsilon$-approximation of $\kappa \cdot uv$ over a triangle $T'$ in $(u,v)$-coordinates, then $\tilde{\ell}(x,y) := \ell(\Phi(x,y))$ is a linear $|\kappa|\varepsilon$-approximation of $G$ over $T = \Phi^{-1}(T')$ in $(x,y)$-coordinates. The $L_\infty$ error scales by $|\kappa|$ because the quadratic part transforms as $G(x,y) = \kappa \cdot uv + \text{(linear)}$, and linear terms contribute only to vertex deviations. Triangle areas scale by $1/J$, so triangle density scales by $J$. Combining these scaling factors, optimal density results for $F(x,y) = xy$ extend to general $G$ with the appropriate factor.
\end{proof}

By Lemma~\ref{lem:reduction-standard}, we assume $F(x,y) = xy$ throughout the remainder of this paper. All density results for the standard form extend to general indefinite quadratics via the appropriate scaling factor.

The standard form $F(x,y) = xy$ has a special anti-symmetry property that enables additional simplifications:

\begin{proposition}[Invariance under axis reflection]
\label{prop:reflection}
Let $T \subset \R^2$ be a triangle and let $\ell$ be a linear $\varepsilon$-approximation of $F(x,y) = xy$ over $T$. Then there exist linear $\varepsilon$-approximations of $F$ over the axis-reflected triangles:
\begin{itemize}
    \item $T_x := \{(x,-y) : (x,y) \in T\}$ (reflection across the $x$-axis), and
    \item $T_y := \{(-x,y) : (x,y) \in T\}$ (reflection across the $y$-axis).
\end{itemize}
\end{proposition}
\begin{proof}
Let $\ell(x,y) = \alpha x + \beta y + \gamma$. For $T_x$, define $\ell_x(x,y) := -\ell(x,-y) = -\alpha x + \beta y - \gamma$. Using the anti-symmetry $F(x,-y) = -F(x,y)$: for $(x',y') = (x,-y) \in T_x$ with $(x,y) \in T$,
\[
\ell_x(x',y') - F(x',y') = -\ell(x',-y') - F(x',y') = -\ell(x,y) + F(x,y) = -[\ell(x,y) - F(x,y)].
\]
Thus $|\ell_x - F|$ over $T_x$ equals $|\ell - F|$ over $T$. The case for $T_y$ is analogous (define $\ell_y(x,y) := -\ell(-x,y) = \alpha x - \beta y - \gamma$).
\end{proof}

\subsection{Tiling the plane with a single triangle}
\label{sec:tiling}
We now show how to derive $\varepsilon$-optimal triangulations from a single $\varepsilon$-optimal triangle.
An $\varepsilon$-optimal triangle is a triangle that has maximum area while admitting a linear $\varepsilon$-approximation (possibly with additional constraints like continuity or one-sidedness).

By Lemma~\ref{lem:invariance_trans}, we can translate any triangle $T$ with a linear $\varepsilon$-approximation arbitrarily in the plane while preserving the approximation error.
By Lemma~\ref{lem:invariance_reflection}, we can also obtain a linear $\varepsilon$-approximation over the point-reflection $T' = -T$.
Since $T$ and $T'$ form a parallelogram, and parallelograms tile the plane, the entire plane can be triangulated using translations of $T$ and $T'$. If $T$ has area $A(T)$, this triangulation has density $1/A(T)$. See \cref{fig:translate-reflect-tile} for an illustration.
\begin{lemma}[\cite{Atariah:2018}]
    \label{lem:area_density}
    Let $T$ be a triangle with area $A(\varepsilon)$ and let $f: \R^2 \to \R$ be a linear approximation of $F$ with an approximation error of $\varepsilon>0$. %
    Then there exists a \pwl approximation of $F$ over the plane with an approximation error of $\varepsilon$
    and a triangle density $\nicefrac{1}{A(\varepsilon)}$.
\end{lemma}
\begin{proof}
Let $f:\R^2 \to \R$ with $f(x,y)=\alpha x + \beta y + \gamma$ be a linear approximation of $F$ such that $\varepsilon_{f,F}(T)\leq \varepsilon$. 
According to Lemma~\ref{lem:invariance_reflection}, we can define a linear function $f':\R^2\to\R$ by $f'(x,y)=-\alpha x - \beta y + \gamma$ such that $\varepsilon_{f',F}(T')\leq \varepsilon$ with $T'\define \{(-x,-y) \in \R^2: (x,y) \in T\}$.
As a consequence, we have the same maximum approximation error over $T$ and $T'$.
According to Lemma~\ref{lem:invariance_trans}, we can translate copies of $T$ and $T'$ arbitrarily in the plane while maintaining the approximation error. 
Since $T'$ is a rotation of $T$ by 180 degrees about the origin, we can shift $T'$ to form a parallelogram with $T$.
Thus we can tile $\R^2$ with copies of $T$ and $T'$. 
Since $T$ and $T'$ each have area $A$, we obtain an $\varepsilon$-triangulation of $\R^2$ with density $\tfrac{1}{A}$.
\end{proof}

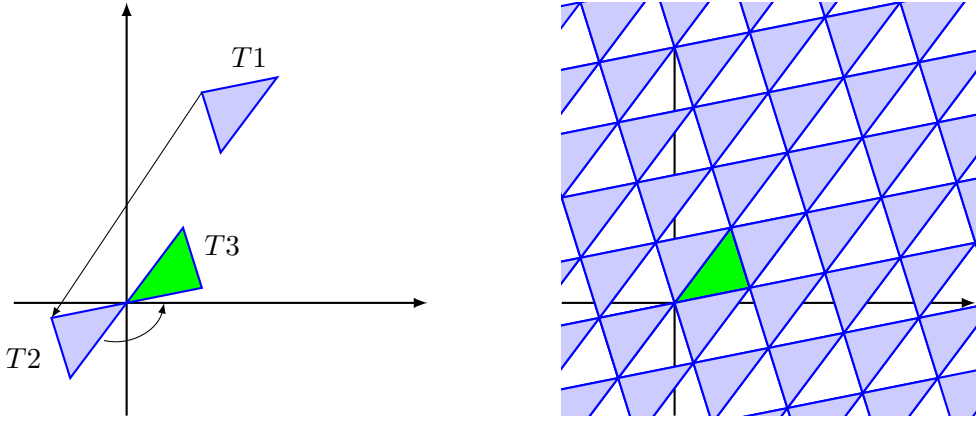
\begin{figure}[h]
\centering
\begin{tikzpicture}[>=latex, mytriangle/.style={thick, draw=blue, fill=blue!20}]

  \draw[thick,->] (-1.5,0) -- (4,0) node[anchor=north west] {};
  \draw[thick,->] (0,-1.5) -- (0,4) node[anchor=south east] {};

  \def\xOne{0} %
  \def\yOne{0} %
  \def\xTwo{1} %
  \def\yTwo{.2} %
  \def\xThree{0.75} %
  \def\yThree{1} %

  \def\t{2} %
  \def\r{3} %

  \draw[mytriangle, fill = green] (\xOne,\yOne) -- (\xTwo,\yTwo) -- (\xThree,\yThree) -- cycle;
  \node at (\xOne+0.9,\yOne+0.5) [above right] {$T3$};

  \draw[mytriangle] (-\xOne,-\yOne) -- (-\xTwo,-\yTwo) -- (-\xThree,-\yThree) -- cycle;
  \node at (-\xOne-1,-\yOne-1) [above left] {$T2$};

  \draw[mytriangle] (-\xOne+\t,-\yOne+\r) -- (-\xTwo+\t,-\yTwo+\r) -- (-\xThree+\t,-\yThree+\r) -- cycle;
  \node at (-\xOne+\t,-\yOne+\r) [above left] {$T1$};

\draw[->](-\xTwo+\t,-\yTwo+\r) -- (-\xTwo,-\yTwo);
\draw[->] (-0.3,-0.5) to [bend right=45] (0.5,0.0);
\end{tikzpicture}%
\hspace{1.5cm}%
\begin{tikzpicture}[>=latex, mytriangle/.style={thick, draw=blue, fill=blue!20}]

  \draw[thick,->] (-1.5,0) -- (4,0) node[anchor=north west] {};
  \draw[thick,->] (0,-1.5) -- (0,4) node[anchor=south east] {};

  \def\xOne{0} %
  \def\yOne{0} %
  \def\xTwo{1} %
  \def\yTwo{.2} %
  \def\xThree{0.75} %
  \def\yThree{1} %

  \def\t{\xTwo+\xThree} %
  \def\r{\yTwo+\yThree} %

  \draw[mytriangle, fill = green] (\xOne,\yOne) -- (\xTwo,\yTwo) -- (\xThree,\yThree) -- cycle;
  \node at (\xOne+0.9,\yOne+0.5) [above right] {$T$};

  \draw[mytriangle] (-\xOne+\t,-\yOne+\r) -- (-\xTwo+\t,-\yTwo+\r) -- (-\xThree+\t,-\yThree+\r) -- cycle;
  
  \draw[mytriangle] (-\xOne,-\yOne) -- (-\xTwo,-\yTwo) -- (-\xThree,-\yThree) -- cycle;

  \clip (-1.5,-1.5) rectangle (4,4);

  \foreach \k in {-5,...,5}{
    \foreach \p in {-5,...,5}{
      \draw[mytriangle] 
        ($(-\xOne,-\yOne)+\k*(\xTwo,\yTwo)+\p*(\xThree,\yThree)$) -- 
        ($(-\xTwo,-\yTwo)+\k*(\xTwo,\yTwo)+\p*(\xThree,\yThree)$) -- 
        ($(-\xThree,-\yThree)+\k*(\xTwo,\yTwo)+\p*(\xThree,\yThree)$) -- cycle;
      
      \draw[mytriangle] 
        ($(-\xOne+\t,-\yOne+\r)+\k*(\xTwo,\yTwo)+\p*(\xThree,\yThree)$) -- 
        ($(-\xTwo+\t,-\yTwo+\r)+\k*(\xTwo,\yTwo)+\p*(\xThree,\yThree)$) -- 
        ($(-\xThree+\t,-\yThree+\r)+\k*(\xTwo,\yTwo)+\p*(\xThree,\yThree)$) -- cycle;
    }
  }
\end{tikzpicture}
\caption{Left: We show how to translate and then possibly reflect a triangle to get a triangle touching the origin and in the first orthant.  Right: We show that with that triangle, we can tile the plane.}
\label{fig:translate-reflect-tile}

\end{figure}

\begin{remark}
\label{rem:continuity-deviations}
    Although Lemma~\ref{lem:area_density} guarantees that we can construct a triangulation with a constant maximum approximation error from any given triangle, it is not guaranteed that this approximation will be continuous. We address this further in Remark~\ref{rem:continuous-constraint}.
\end{remark}

We now summarize the reductions established so far for the two-dimensional case. A \emph{parallelogram tiling} is a triangulation of $\R^2$ consisting of translated copies of a triangle $T$ and its point-reflection $T' = -T$, arranged so that $T$ and $T'$ form a parallelogram that tiles the plane by translation.

\begin{corollary}[Normalized form for optimal triangles]
\label{cor:normalized-triangle}

By Lemma~\ref{lem:area_density}, finding an $\varepsilon$-optimal triangulation for $F(x,y) = xy$ reduces to finding a single triangle $T$ of maximum area that admits a linear $\varepsilon$-approximation. This reduction applies both to general \pwl approximations and to continuous \pwl approximations. Furthermore:
\begin{enumerate}
    \item By Lemma~\ref{lem:invariance_trans}, we may assume one vertex of $T$ is at the origin, \ie $v_1=(x_1,y_1)=(0,0)$.
    \item By Lemma~\ref{lem:invariance_reflection} and Proposition~\ref{prop:reflection}, if the other two vertices lie in the same quadrant (both in Q1 or both in Q3), we may normalize so that both lie in the closed first quadrant. (Vertices in Q2 or Q4 are also possible and would lead to an equivalent ``HLL'' deviation pattern; the analysis is symmetric.)
\end{enumerate}
Thus, the search for optimal triangulations reduces to maximizing the area of a triangle with one vertex at the origin, subject to the constraint that the triangle admits a linear $\varepsilon$-approximation.
\end{corollary}

This completes the reduction. In the remainder of the paper, we analyze the error structure of a single triangle (Section~\ref{sec:error-analysis}), formulate and solve the resulting optimization problem (Section~\ref{sec:solution}), and prove the main theorems (Section~\ref{sec:proofs}).

\section{Error analysis of a single triangle}
\label{sec:error-analysis}
Having reduced the global optimization problem to finding a single optimal triangle, we now analyze how the approximation error behaves over a triangle. This analysis reveals a key structure: the maximum error is always attained on an edge, and the error along each edge admits a simple closed-form expression. We also introduce \emph{pseudo-Euclidean motion}, a transformation that preserves both area and approximation error, which enables further normalization.

\paragraph{Setup and notation}
Throughout this section, we work with the standard form $F(x,y) = xy$.
Let $T$ be a triangle with vertices $v_1 = (x_1, y_1)$, $v_2 = (x_2, y_2)$, and $v_3 = (x_3, y_3)$. A linear approximation $\ell\colon \R^2 \to \R$ of $F$ over $T$ is characterized by its values at the vertices.

\paragraph{Deviations}
We define the \emph{deviation} at vertex $v_i$ as
\[
d_i := \ell(v_i) - F(v_i) = \ell(x_i, y_i) - x_i y_i.
\]
The deviation $d_i$ measures the signed vertical distance between the approximating plane and the surface $F$ at vertex $v_i$: positive deviations correspond to overestimation, negative to underestimation. Figure~\ref{fig:deviation-3d} illustrates this concept.

\begin{remark}[Continuity constraint]
\label{rem:continuous-constraint}
If we require the \pwl approximation to be \emph{continuous}, then in the parallelogram tiling from Lemma~\ref{lem:area_density}, adjacent triangles must have matching linear functions along shared edges. Consider edge $e_3$ connecting $v_2$ to $v_3$: on one side is triangle $T$ with deviations $d_2$ at $v_2$ and $d_3$ at $v_3$; on the other side is a translated copy of $-T$, which by the construction has deviations $d_3$ at $v_2$ and $d_2$ at $v_3$ (the deviations are swapped). For the two linear functions to agree along the entire edge, we must have $d_2 = d_3$. By similar reasoning for edges $e_1$ and $e_2$, continuity of the \pwl approximation requires
\[
d_1 = d_2 = d_3.
\]
Thus, for continuous approximations, the deviations must be constant across all vertices. We prove in \cref{thm:optimal-triangulation-cont} that this constraint leads to optimal density $\tfrac{\sqrt{3}}{8\varepsilon}$ among parallelogram tilings, and we conjecture (\cref{conj:continuous-optimality}) that no continuous triangulation---parallelogram-based or otherwise---can achieve density better than this constant-deviation optimum.
\end{remark}

\begin{figure}[ht]
    \centering
    
\newcommand{\devpanelcolor}[6]{%
\begin{tikzpicture}
\begin{axis}[
    view={80}{35},
    xlabel={$x$},
    ylabel={$y$},
    zlabel={$z$},
    xmin=-0.7, xmax=3.2,
    ymin=-0.3, ymax=3.2,
    zmin=-2.0, zmax=5.5,
    width=6.2cm,
    height=5.0cm,
    grid=major,
    grid style={gray!20},
    axis lines=center,
    axis line style={thick, -Stealth},
    tick label style={font=\tiny},
    label style={font=\scriptsize},
    ztick={0,2,4},
    xtick={1,2,3},
    ytick={1,2,3},
]
\addplot3[surf, opacity=0.30, colormap={bluegray}{color=(blue!20) color=(blue!50)},
    shader=flat, samples=18, domain=0:3, y domain=0:3, z buffer=sort] {x*y};
\pgfmathsetmacro{\xA}{0}  \pgfmathsetmacro{\yA}{0}  \pgfmathsetmacro{\dA}{#1}
\pgfmathsetmacro{\xB}{2.3} \pgfmathsetmacro{\yB}{0.9} \pgfmathsetmacro{\dB}{#2}
\pgfmathsetmacro{\xC}{0.9} \pgfmathsetmacro{\yC}{2.3} \pgfmathsetmacro{\dC}{#3}
\pgfmathsetmacro{\FA}{\xA*\yA} \pgfmathsetmacro{\FB}{\xB*\yB} \pgfmathsetmacro{\FC}{\xC*\yC}
\pgfmathsetmacro{\zA}{\FA+\dA} \pgfmathsetmacro{\zB}{\FB+\dB} \pgfmathsetmacro{\zC}{\FC+\dC}
\addplot3[fill=gray!15, opacity=0.4, draw=gray!50] coordinates {(\xA,\yA,0) (\xB,\yB,0) (\xC,\yC,0) (\xA,\yA,0)};
\addplot3[thin, dotted, gray!50] coordinates {(\xA,\yA,0) (\xA,\yA,\zA)};
\addplot3[thin, dotted, gray!50] coordinates {(\xB,\yB,0) (\xB,\yB,\zB)};
\addplot3[thin, dotted, gray!50] coordinates {(\xC,\yC,0) (\xC,\yC,\zC)};
\addplot3[fill=#4!35, opacity=0.8, draw=#4!70!black, very thick]
  coordinates {(\xA,\yA,\zA) (\xB,\yB,\zB) (\xC,\yC,\zC) (\xA,\yA,\zA)};
\addplot3[only marks, mark=*, mark size=2.5pt, color=blue!70!black]
  coordinates {(\xA,\yA,\FA) (\xB,\yB,\FB) (\xC,\yC,\FC)};
\addplot3[only marks, mark=square*, mark size=2.5pt, color=#4!80!black]
  coordinates {(\xA,\yA,\zA) (\xB,\yB,\zB) (\xC,\yC,\zC)};
\addplot3[very thick, color=#5, {Stealth[length=1.5mm]}-{Stealth[length=1.5mm]}]
  coordinates {(\xA,\yA,\FA) (\xA,\yA,\zA)};
\addplot3[very thick, color=#5, {Stealth[length=1.5mm]}-{Stealth[length=1.5mm]}]
  coordinates {(\xB,\yB,\FB) (\xB,\yB,\zB)};
\addplot3[very thick, color=#5, {Stealth[length=1.5mm]}-{Stealth[length=1.5mm]}]
  coordinates {(\xC,\yC,\FC) (\xC,\yC,\zC)};
\node[color=#5, font=\scriptsize\bfseries] at (axis cs:{\xA-0.35},{\yA-0.15},{\FA+\dA/2}) {$d_1$};
\node[color=#5, font=\scriptsize\bfseries] at (axis cs:{\xB+0.05},{\yB-0.35},{\FB+\dB/2}) {$d_2$};
\node[color=#5, font=\scriptsize\bfseries] at (axis cs:{\xC+0.45},{\yC+0.1},{\FC+\dC/2}) {$d_3$};
\node[blue!60!black, font=\scriptsize] at (axis cs:2.5,2.5,4.5) {$F$};
\node[#4!70!black, font=\scriptsize] at (axis cs:1.5,0.6,{(\zA+\zB)/2+0.3}) {$L$};
\end{axis}
\node[below, font=\small] at (current bounding box.south) {#6};
\end{tikzpicture}%
}

\begin{tabular}{@{}ccc@{}}
\devpanelcolor{1.8}{1.2}{1.2}{green}{green!50!black}{(a) $d_i > 0$: overestimation} &
\devpanelcolor{-1.5}{-1.0}{-1.0}{red}{red!70!black}{(b) $d_i < 0$: underestimation} &
\devpanelcolor{1.5}{-0.8}{1.0}{orange}{orange!70!black}{(c) Mixed $d_i$: general}
\end{tabular}

    \caption{The surface $F(x,y) = xy$ (blue mesh) and linear approximations $L$ with different deviation patterns. The deviations $d_i$ measure the signed vertical distance from surface to approximation at each vertex: (a)~all positive (green, overestimation), (b)~all negative (red, underestimation), (c)~mixed signs (orange, general approximation).}
    \label{fig:deviation-3d}
\end{figure}
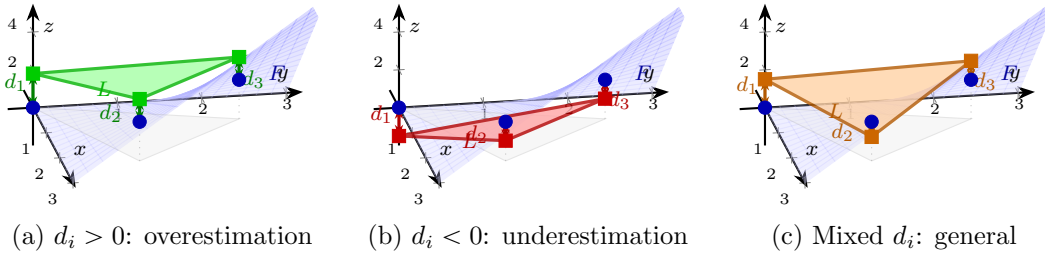

\subsection{Maximum error on edges}

Consider an edge $e$ of triangle $T$ with endpoints $(x_1, y_1)$ and $(x_2, y_2)$ having deviations $d_1$ and $d_2$. Writing the edge displacement as $(\Delta x, \Delta y) = (x_2 - x_1, y_2 - y_1)$, we define the \emph{edge product}
\[
k := \Delta x \cdot \Delta y = (x_2 - x_1)(y_2 - y_1) = F(\Delta x, \Delta y),
\]
which is simply the function $F(x,y) = xy$ evaluated on the edge vector. The edge product measures the curvature of $F$ along the edge: it is positive when the edge vector points into the first or third quadrant ($F$ curves upward), negative when it points into the second or fourth quadrant ($F$ curves downward), and zero for axis-parallel edges (along which $F$ is linear). We call the edge \emph{ascending} if $k > 0$, \emph{descending} if $k < 0$, and \emph{axis-parallel} if $k = 0$. When we wish to emphasize the dependence on parameters, we write $E^*(k, d_a, d_b)$ for the extremal error on an edge with product $k$ and endpoint deviations $d_a, d_b$. For edge $e_i$ we write $E_i^*$ for its extremal error.

\begin{lemma}[Maximum error on edges]
\label{lem:max_error}
\begin{enumerate}[label=\textup{(\roman*)}]
\item[]
\item \label{lem:max_error:boundary} The maximum approximation error over $T$ is attained on an edge.
\item \label{lem:max_error:formula} Parameterizing an edge $e$ by $\lambda \in [0,1]$, the error along $e$ is
\begin{equation}
\label{eq:edge-error-lambda}
E(\lambda) = (1-\lambda)d_1 + \lambda d_2 + \lambda(1-\lambda)k.
\end{equation}
If $k = 0$, then $E(\lambda)$ is linear and the maximum of $|E|$ on $e$ is $\max\{|d_1|, |d_2|\}$.
\item \label{lem:max_error:extremum} If $k \neq 0$, then $E(\lambda)$ is a quadratic with interior extremum at
\begin{equation}
\label{eq:lambda-star}
\lambda^* = \frac{d_2 - d_1}{2k} + \frac{1}{2},
\end{equation}
and when $\lambda^* \in (0,1)$, the extremal error value is
\begin{equation}
\label{eq:edge-error-star}
E^* := E(\lambda^*) = \frac{d_1 + d_2}{2} + \frac{k}{4} + \frac{(d_2-d_1)^2}{4k}.
\end{equation}
Equivalently:
\begin{equation}
\label{eq:edge-error-expanded}
E^* = \frac{(d_1 - d_2 + k)^2 + 4d_2\, k}{4k} = \frac{d_1^2 - 2d_1(d_2 - k) + (d_2 + k)^2}{4k}.
\end{equation}
\end{enumerate}
\end{lemma}

\begin{proof}
The function $F(x,y) = xy$ is linear along every axis-parallel line: for fixed $y_0$, we have $F(x, y_0) = y_0 x$, and for fixed $x_0$, we have $F(x_0, y) = x_0 y$. Since $\ell$ is also linear, the error $E(x,y) = \ell(x,y) - F(x,y)$ is linear along every axis-parallel line.

Now suppose the maximum of $|E|$ over $T$ occurs at an interior point $p$. Moving from $p$ along a horizontal or vertical line until hitting the boundary of $T$, the function $|E|$ restricted to this segment is convex (as the absolute value of a linear function). Hence the maximum on this segment is at an endpoint, which lies on the boundary. This contradicts $p$ being the maximizer, so the maximum must occur on an edge. (This argument appears in \cite{hager:2022}.)

\textbf{Error formula.} Parameterize the edge as $(x,y) = (1-\lambda)(x_1, y_1) + \lambda(x_2, y_2)$ for $\lambda \in [0,1]$. Then:
\begin{align*}
F|_e(\lambda) &= \bigl((1-\lambda)x_1 + \lambda x_2\bigr)\bigl((1-\lambda)y_1 + \lambda y_2\bigr) \\
&= (1-\lambda)^2 x_1y_1 + \lambda(1-\lambda)(x_1y_2 + x_2y_1) + \lambda^2 x_2y_2, \\
\ell|_e(\lambda) &= (1-\lambda)(x_1y_1 + d_1) + \lambda(x_2y_2 + d_2).
\end{align*}
Computing the error $E(\lambda) = \ell|_e(\lambda) - F|_e(\lambda)$:
\begin{align*}
E(\lambda) &= (1-\lambda)d_1 + \lambda d_2 + (1-\lambda)x_1y_1 + \lambda x_2y_2 \\
&\quad - (1-\lambda)^2 x_1y_1 - \lambda(1-\lambda)(x_1y_2 + x_2y_1) - \lambda^2 x_2y_2 \\
&= (1-\lambda)d_1 + \lambda d_2 + \lambda(1-\lambda)\bigl(x_1y_1 + x_2y_2 - x_1y_2 - x_2y_1\bigr) \\
&= (1-\lambda)d_1 + \lambda d_2 + \lambda(1-\lambda)(x_2-x_1)(y_2-y_1),
\end{align*}
which gives \eqref{eq:edge-error-lambda}. Note that $E(0) = d_1$ and $E(1) = d_2$, so the formula interpolates between the vertex deviations. Setting $E'(\lambda) = d_2 - d_1 + (1-2\lambda)k = 0$ yields \eqref{eq:lambda-star}. Substituting back and simplifying yields \eqref{eq:edge-error-star}--\eqref{eq:edge-error-expanded}.
\end{proof}

\begin{remark}[Geometric interpretation of edge error]
\label{rem:edge-error-geometric}
The error formula \eqref{eq:edge-error-lambda} has a natural geometric interpretation. The term $(1-\lambda)d_1 + \lambda d_2$ represents linear interpolation of the vertex deviations, while the quadratic term $\lambda(1-\lambda)k$ captures the \emph{curvature} of $F(x,y) = xy$ along the edge.

To see this, consider an edge from $(x_1, y_1)$ to $(x_2, y_2)$ parameterized as $(x(\lambda), y(\lambda)) = (1-\lambda)(x_1, y_1) + \lambda(x_2, y_2)$. The function $F$ restricted to this edge is:
\[
F|_e(\lambda) = x(\lambda) \cdot y(\lambda) = (1-\lambda)^2 x_1y_1 + \lambda(1-\lambda)(x_1y_2 + x_2y_1) + \lambda^2 x_2y_2,
\]
which is quadratic in $\lambda$ (unless the edge is axis-parallel, in which case $k = 0$). The linear approximation $\ell|_e(\lambda) = (1-\lambda)(x_1y_1 + d_1) + \lambda(x_2y_2 + d_2)$ interpolates linearly between the vertex values. The difference $E(\lambda) = \ell|_e(\lambda) - F|_e(\lambda)$ therefore contains a quadratic term proportional to the ``curvature'' of $F$ along the edge, which is measured by the edge product $k = (x_2 - x_1)(y_2 - y_1)$.

The factor $\lambda(1-\lambda)$ is maximized at $\lambda = 1/2$, reflecting that the maximum deviation from linearity occurs at the midpoint of the edge or its endpoints. This geometric insight explains why optimal triangles balance the edge products $k$ against the vertex deviations $d_i$ to minimize the maximum error.
\end{remark}

\subsection{Pseudo-Euclidean motion}
\label{sec:pe-motion}

The saddle surface $z = xy$ possesses a continuous family of symmetries beyond the discrete symmetries (translation, reflection) discussed in Section~3. The level curves $xy = c$ are hyperbolas, and there exist area-preserving transformations that slide points along these hyperbolas while preserving the function value. These transformations are called \emph{pseudo-Euclidean motions} because they are isometries of the pseudo-Euclidean plane---the plane equipped with the indefinite metric $ds^2 = dx\,dy$ rather than the Euclidean metric $ds^2 = dx^2 + dy^2$. Just as Euclidean rotations preserve circles ($x^2 + y^2 = c$), pseudo-Euclidean motions preserve hyperbolas ($xy = c$).

This concept appears in the study of indefinite quadratic forms and saddle surfaces; see Pottmann et al.\ \cite{Pottmann:2000} for applications to surface approximation. For our purposes, pE-motion provides an additional degree of freedom to normalize optimal triangles: if two vertices lie on the same hyperbola, we can transform the triangle to a symmetric configuration without changing its area or approximation error.

\begin{lemma}[Pseudo-Euclidean motion preserves area and error]
\label{lem:pe-motion}
The \emph{pseudo-Euclidean motion} (pE-motion) with parameter $m \neq 0$ is the linear transformation
\begin{equation}
\label{eq:pe-motion}
(x,y) \mapsto (mx, \tfrac{y}{m}).
\end{equation}
Let $T$ be a triangle with vertices $(x_1,y_1)$, $(x_2,y_2)$, $(x_3,y_3)$ and let $\ell$ be a linear approximation of $F(x,y) = xy$ over $T$ with deviations $d_1, d_2, d_3$. Let $T'$ be the image of $T$ under a pE-motion. Then:
\begin{enumerate}
    \item $T'$ has the same area as $T$.
    \item There exists a linear approximation $\ell'$ of $F$ over $T'$ with the same deviations $d_1, d_2, d_3$ and the same maximum approximation error.
\end{enumerate}
\end{lemma}
\begin{proof}
The transformation matrix $W = \begin{bmatrix} m & 0 \\ 0 & \tfrac{1}{m} \end{bmatrix}$ has $\det(W) = 1$, so it preserves area.

For the error, note that $F(mx, \tfrac{y}{m}) = mx \cdot \tfrac{y}{m} = xy = F(x,y)$. Since the deviations $d_i = \ell(v_i) - F(v_i)$ are defined at vertices and pE-motion preserves $F$, the same deviations apply to $T'$ if we define $\ell'$ appropriately.

Explicitly, by Lemma~\ref{lem:invariance_trans} we may assume $(x_1,y_1)=(0,0)$. By Lemma~\ref{lem:max_error}, the error along any edge depends only on the endpoint deviations $d_i, d_j$ and the edge product $k_{ij} = (x_j - x_i)(y_j - y_i)$: the critical parameter is $\lambda^* = \frac{d_j - d_i}{2k_{ij}} + \frac{1}{2}$ and the extremal error is $E^* = (1-\lambda^*)d_i + \lambda^* d_j + \lambda^*(1-\lambda^*)k_{ij}$.

Under simultaneous pE-motion $(x,y) \mapsto (mx, y/m)$ applied to all vertices, the edge product of \emph{every} edge is preserved:
\[
k'_{ij} = (mx_j - mx_i)\bigl(\tfrac{y_j}{m} - \tfrac{y_i}{m}\bigr) = (x_j - x_i)(y_j - y_i) = k_{ij}.
\]
Since both the deviations $d_i$ and all edge products $k_{ij}$ are unchanged, the formulas above show that $\lambda^*$ and $E^*$ are preserved on every edge---ascending and descending alike. Therefore the maximum approximation error over the entire triangle is preserved.
\end{proof}

\begin{corollary}[Normalization via pE-motion]
\label{cor:pe-normalization}
Let $T$ be a triangle with one vertex at the origin and the other two vertices $(x_2,y_2)$ and $(x_3,y_3)$ satisfying $x_2y_2 = x_3y_3 > 0$. Then there exists a pE-motion parameter $m > 0$ such that the transformed vertices satisfy $(x_2', y_2') = (y_3', x_3')$.
\end{corollary}
\begin{proof}
We need $mx_2 = \tfrac{y_3}{m}$ and $\tfrac{y_2}{m} = mx_3$. From the first equation, $m^2 = \tfrac{y_3}{x_2}$. Substituting into the second: $\tfrac{y_2 x_2}{y_3} = \tfrac{y_3 x_3}{x_2}$, which gives $x_2 y_2 = x_3 y_3$. This holds by assumption, so $m = \sqrt{\tfrac{y_3}{x_2}} = \sqrt{\tfrac{y_2}{x_3}}$ works.
\end{proof}

Geometrically, pE-motion slides vertices along hyperbolas $xy = c$ while preserving triangle area. Figure~\ref{fig:pe-motion} illustrates this.

\begin{remark}[Optimal vertices lie on hyperbolas]
\label{rem:hyperbola-vertices}
A geometric property of the optimal solution, which emerges from the analysis in Section~\ref{sec:elem-proof}, is that \emph{optimal vertices lie on hyperbolas} $xy = \text{constant}$. Specifically, for an optimal triangle with vertices $(0,0)$, $(x_2, y_2)$, and $(x_3, y_3)$, the non-origin vertices satisfy $x_2 y_2 = x_3 y_3 = k^*$, where $k^*$ is the optimal edge product. By pE-motion (Corollary~\ref{cor:pe-normalization}), we can then normalize to the symmetric form $(x_2, y_2) = (y_3, x_3)$, with both vertices on the same hyperbola branch in the first quadrant.

For the \emph{continuous case}, $k^* = \tfrac{8\varepsilon}{3}$, so vertices lie on $xy = \tfrac{8\varepsilon}{3}$. For the \emph{general case}, $k^* = \tfrac{32\varepsilon}{9}$, corresponding to $xy = \tfrac{32\varepsilon}{9}$. This provides geometric intuition: pE-motion slides vertices along their hyperbola while preserving area and error structure.
\end{remark}

\begin{remark}[Properties of pE-motion]
\label{rem:pe-properties}
When one vertex is at the origin, pE-motion has useful properties for understanding optimal triangles:
\begin{enumerate}
\item \emph{Single-vertex pE-motion}: Applying pE-motion to $v_2$ alone (keeping $v_3$ fixed) preserves both $k_1 = x_2 y_2$ and $k_2 = x_3 y_3$. Hence $E_1^*$ and $E_2^*$ are unchanged, while only $E_3^*$ varies. This provides a 1-parameter family affecting only the descending edge constraint.
\item \emph{Simultaneous pE-motion}: Applying the same pE-motion $(x,y) \mapsto (mx, y/m)$ to both $v_2$ and $v_3$ preserves all three edge products: $k_1$, $k_2$, and $k_3 = (x_3-x_2)(y_3-y_2) = m(x_3-x_2) \cdot \tfrac{1}{m}(y_3-y_2)$. Hence all $E_i^*$ are preserved. This is used in Corollary~\ref{cor:pe-normalization} to normalize without changing approximation quality.
\end{enumerate}
\end{remark}

\begin{figure}[h]
    \centering
\begin{tikzpicture}
\begin{axis}[
  axis lines=middle,
  axis line style={thick,black},
  xmin=-1.5, xmax=6.5,
  ymin=-1.2, ymax=5.5,
  width=14cm, height=9cm,
  grid=both,
  major grid style={draw=gray!55, line width=0.35pt},
  minor grid style={draw=gray!20, line width=0.20pt},
  minor tick num=4,
  ticklabel style={font=\small},
  clip=true,
]
\pgfmathsetmacro{\epsnum}{1}
\pgfmathsetmacro{\xA}{1}
\pgfmathsetmacro{\xB}{2+sqrt(3)}        %
\pgfmathsetmacro{\yA}{32*\epsnum/(9*\xA)}
\pgfmathsetmacro{\yB}{32*\epsnum/(9*\xB)}
\pgfmathsetmacro{\xOa}{1.3}
\pgfmathsetmacro{\xOb}{4.9}
\pgfmathsetmacro{\yOa}{32*\epsnum/(9*\xOa)}
\pgfmathsetmacro{\yOb}{32*\epsnum/(9*\xOb)}
\addplot[draw=none, fill=blue!25, fill opacity=0.25]
  coordinates {(0,0) (\xA,\yA) (\xB,\yB)} -- cycle;
\addplot[draw=none, fill=orange!30, fill opacity=0.25]
  coordinates {(0,0) (\xOa,\yOa) (\xOb,\yOb)} -- cycle;
\addplot[domain=0.6:6.5, samples=250, very thick, color=violet]
  {32*\epsnum/(9*x)};
\addplot[domain=0.6:6.5, samples=250, very thick, color=gray!70]
  {16*\epsnum/(3*x)};
\addplot[line width=1.2pt, color=blue!70!black]
  coordinates {(0,0) (\xA,\yA) (\xB,\yB)} -- cycle;
\addplot[line width=1.2pt, color=orange!85!black]
  coordinates {(0,0) (\xOa,\yOa) (\xOb,\yOb)} -- cycle;
\node[color=violet, font=\large] at (axis cs:2.25,4.15)
  {$f(x)=\dfrac{32\varepsilon}{9x}$};
\node[color=gray!70, font=\large] at (axis cs:3.95,2.45)
  {$g(x)=\dfrac{16\varepsilon}{3x}$};
\end{axis}
\end{tikzpicture}
    \caption{Illustration of pseudo-Euclidean motion: the transformation $(x,y) \mapsto (mx, y/m)$ preserves both area and the product $xy$, which determines the approximation error. 
    Triangles with vertices on the hyperbola $f(x) = \frac{32\varepsilon}{9x}$ have the same area regardless of their shape. For the violet triangle $(x_2,y_2)=(y_3,x_3)$ holds. The midpoints of both triangles also move along a shared hyperbola.}
    \label{fig:pe-motion}
\end{figure}
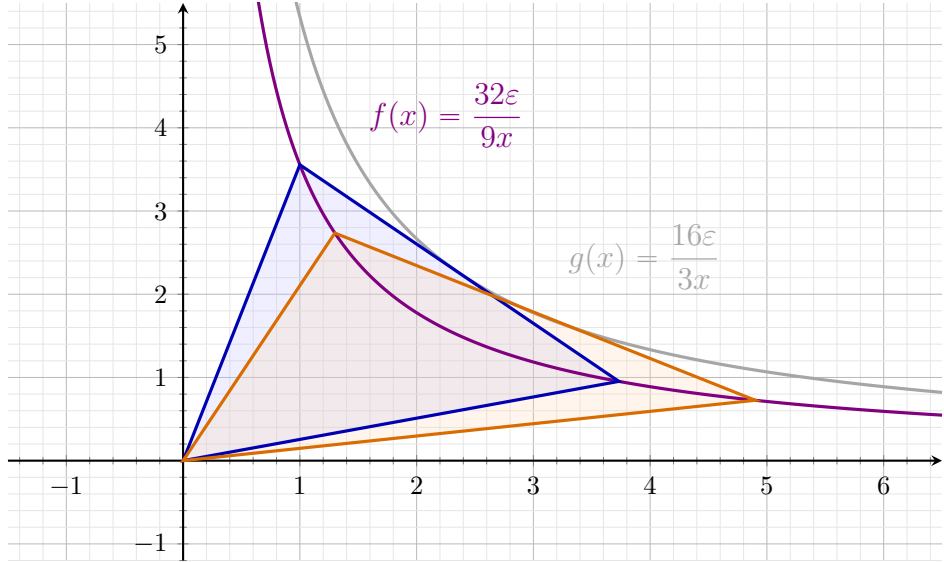

Combining the results of Sections~\ref{sec:reduction} and \ref{sec:error-analysis}, we have established that:
\begin{enumerate}
    \item Finding optimal triangulations reduces to finding a single triangle of maximum area (Corollary~\ref{cor:normalized-triangle}).
    \item We may assume the triangle has one vertex at the origin.
    \item The maximum error over the triangle is attained on an edge (Lemma~\ref{lem:max_error}\ref{lem:max_error:boundary}), with the error along each edge given by Lemma~\ref{lem:max_error}\ref{lem:max_error:formula}, equation~\eqref{eq:edge-error-lambda}.
\end{enumerate}
The pE-motion of Section~\ref{sec:pe-motion} provides an additional tool: once the optimal edge products are determined, Corollary~\ref{cor:pe-normalization} can normalize the triangle to the symmetric form $(x_2, y_2) = (y_3, x_3)$. However, this normalization is not needed for the optimization itself---the proof in the next section works directly with edge products and slack variables.

These reductions transform the infinite-dimensional problem of optimizing over all triangulations into a finite-dimensional one: the only remaining degrees of freedom are the vertex coordinates $(x_2, y_2, x_3, y_3)$ and the three deviations $d_1, d_2, d_3$. We solve this problem in the next section.

\section{Solving the optimal triangle problem}
\label{sec:solution}
We now formulate the optimal triangle problem as a nonlinear optimization problem and derive analytical solutions. The key to obtaining closed-form solutions is a reformulation via slack variables that transforms the problem into an optimization over box-constrained edge products. This is then solved exactly through separate convexity (reducing to box vertices), symmetry arguments, and an algebraic optimality certificate.

\paragraph{Section roadmap.} This section is organized as follows:
\begin{itemize}[nosep]
\item \textbf{Section~\ref{sec:opt-models}}: Formulates the optimization model using deviations and edge products.
\item \textbf{Section~\ref{sec:slack-reformulation}}: Reformulates via slack variables and edge-product bounds, reducing to an optimization over box-constrained edge products.
\item \textbf{Section~\ref{sec:elem-proof}}: Solves the optimization: (1) separate convexity reduces the edge-product variables to finitely many box vertices, (2) a dominance argument shows that the optimum occurs when all edge-product bounds are saturated (Type~A), corresponding to tight interior extremum constraints on each edge, (3) a strict inequality proves symmetry $a_2 = a_3$, and (4) an algebraic optimality certificate yields the closed-form solution.
\item \textbf{Section~\ref{sec:scaling}}: Applies the solution to all approximation types (general, continuous, under/over).
\end{itemize}

\subsection{Optimal triangles as an optimization problem}
\label{sec:opt-models}

We now formulate the optimization problem using the notation from Section~\ref{sec:error-analysis}: a triangle $T$ with vertices $v_1 = (x_1, y_1)$, $v_2 = (x_2, y_2)$, $v_3 = (x_3, y_3)$ and deviations $d_1, d_2, d_3$.
By Corollary~\ref{cor:normalized-triangle}, we may assume $v_1$ is at the origin, so $(x_1, y_1) = (0,0)$.

We label the three edges: $e_1$ connects the origin to $v_2$, $e_2$ connects the origin to $v_3$, and $e_3$ connects $v_2$ to $v_3$. Figure~\ref{fig:deviation-2d} illustrates the setup after normalization (described below).

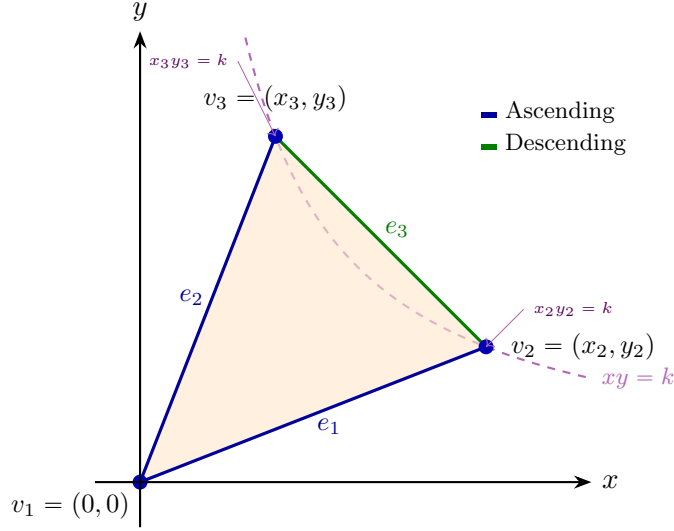
\begin{figure}[ht]
    \centering
\begin{tikzpicture}[scale=2.0, 
    vertex/.style={circle, fill=blue!70!black, inner sep=2pt},
    edge label/.style={font=\small, fill=white, inner sep=1pt}]
\coordinate (V1) at (0, 0);
\coordinate (V2) at (2.3, 0.9);
\coordinate (V3) at (0.9, 2.3);
\pgfmathsetmacro{\k}{2.07}
\draw[domain=0.7:3, samples=100, thick, violet!60, dashed] 
    plot (\x, {\k/\x}) node[right, font=\footnotesize] {$xy = k$};
\fill[orange!20, opacity=0.6] (V1) -- (V2) -- (V3) -- cycle;
\draw[very thick, blue!60!black] (V1) -- (V2) node[edge label, midway, below right] {$e_1$};
\draw[very thick, blue!60!black] (V1) -- (V3) node[edge label, midway, above left] {$e_2$};
\draw[very thick, green!50!black] (V2) -- (V3) node[edge label, midway, above right] {$e_3$};
\node[vertex] at (V1) {};
\node[vertex] at (V2) {};
\node[vertex] at (V3) {};
\node[below left, font=\small] at (V1) {$v_1 = (0, 0)$};
\node[right, font=\small] at ($(V2)+(0.1,0)$) {$v_2 = (x_2, y_2)$};
\node[above, font=\small] at ($(V3)+(0,0.1)$) {$v_3 = (x_3, y_3)$};
\draw[-Stealth, thick] (-0.3, 0) -- (3.0, 0) node[right] {$x$};
\draw[-Stealth, thick] (0, -0.3) -- (0, 3.0) node[above] {$y$};
\node[align=left, font=\footnotesize, anchor=north west] at (2.2, 2.6) {
    \textcolor{blue!60!black}{\rule{6pt}{2pt}} Ascending\\[1pt]
    \textcolor{green!50!black}{\rule{6pt}{2pt}} Descending
};
\draw[<-, thin, violet!60] (V2) -- ++(0.25, 0.25) node[right, font=\tiny, violet!70!black] {$x_2 y_2 = k$};
\draw[<-, thin, violet!60] (V3) -- ++(-0.25, 0.5) node[left, font=\tiny, violet!70!black] {$x_3 y_3 = k$};
\end{tikzpicture}
    \caption{An optimal triangle (after normalization): $v_1$ at the origin, $v_2$ and $v_3$ in the first quadrant with equal edge products $x_2 y_2 = x_3 y_3 = k$, so both lie on the hyperbola $xy = k$. Edges $e_1$ and $e_2$ (blue) are ascending ($k > 0$), while $e_3$ (green) is descending. The optimization in Section~\ref{sec:elem-proof} works with edge products directly and does not assume this geometric normalization.}
    \label{fig:deviation-2d}
\end{figure}

Let $\varepsilon$ be a bound on the maximum approximation error.
By Lemma~\ref{lem:max_error}\ref{lem:max_error:boundary}, the maximum error over $T$ is attained on an edge.
For the optimization, we bound the extremal error $E^*$ on each edge.
Using Lemma~\ref{lem:max_error}\ref{lem:max_error:extremum}, equation~\eqref{eq:edge-error-expanded}, with $(x_1,y_1) = (0,0)$, the edge product for edge $e_1$ from the origin to $(x_2,y_2)$ is simply $k_1 = x_2y_2$, and the extremal error becomes:
\[
E_1^* = \frac{d_1^2 - 2d_1(d_2 - x_2y_2) + (d_2 + x_2y_2)^2}{4x_2y_2}.
\]
Similar expressions hold for edges $e_2$ and $e_3$.

Finally, the objective is to maximize the area of the triangle $T$:
$$A(T)=\frac{1}{2}
\abs{
\det 
\begin{pmatrix}
x_3 & x_2\\
y_3 & y_2 
\end{pmatrix}}
=\tfrac{1}{2}\abs{x_3y_2-x_2y_3}.$$

Combining the error constraints from Lemma~\ref{lem:max_error}\ref{lem:max_error:extremum} with the area objective gives the following optimization model:
\begin{subequations}
\label{model:optimal-triangle}
\begin{align}
    \text{Maximize:}\quad & \tfrac{1}{2}(x_3y_2-x_2y_3) \label{model:obj}\\
    \text{subject to:}\quad  E_{\min} &\leq \frac{d_1^2 - 2 d_1\left(d_2 - x_2 y_2\right) + \left(d_2 + x_2 y_2\right)^2}{4 x_2 y_2} \leq E_{\max}, \label{model:edge1}\\
    E_{\min} &\leq \frac{d_1^2 - 2 d_1\left(d_3 - x_3 y_3\right) + \left(d_3 + x_3 y_3\right)^2}{4 x_3 y_3} \leq E_{\max}, \label{model:edge2}\\
    E_{\min} &\leq \frac{d_2^2 - 2 d_2\left(d_3 - \left(x_3 - x_2\right)\left(y_3 - y_2\right)\right) + \left(d_3 + \left(x_3 - x_2\right)\left(y_3 - y_2\right)\right)^2}{4\left(x_3 - x_2\right)\left(y_3 - y_2\right)} \leq E_{\max}, \label{model:edge3}\\
    & \quad x_2,y_2,x_3,y_3 \in \R, \quad d_1,d_2,d_3 \in [E_{\min}, E_{\max}]. \label{model:bounds}
\end{align}
\end{subequations}
Constraints \eqref{model:edge1}--\eqref{model:edge3} bound the extremal error $E^*$ on each edge using Lemma~\ref{lem:max_error}\ref{lem:max_error:extremum}, equation~\eqref{eq:edge-error-expanded}. Note that the vertex errors are automatically bounded since $E(0) = d_1$ and $E(1) = d_2$ by Lemma~\ref{lem:max_error}\ref{lem:max_error:formula}, equation~\eqref{eq:edge-error-lambda}, and the deviations satisfy $d_i \in [E_{\min}, E_{\max}]$.

The different approximation types correspond to different parameter choices:
\begin{center}
\begin{tabular}{lcc}
\toprule
Approximation Type & $E_{\min}$ & $E_{\max}$ \\
\midrule
General & $-\varepsilon$ & $\varepsilon$ \\
Underestimation & $-\varepsilon$ & $0$ \\
Overestimation & $0$ & $\varepsilon$ \\
\bottomrule
\end{tabular}
\end{center}

These optimization models can be solved numerically using nonlinear solvers such as Gurobi or BARON, which employ branch-and-bound and convexification techniques. Alternatively, exact solutions can be obtained using computer algebra systems such as Mathematica via cylindrical algebraic decomposition. The latter approach becomes tractable after the problem reformulation in the following subsections. We have verified our analytical results using both approaches: numerical optimization confirms the optimal values to high precision, and symbolic computation in Mathematica confirms the closed-form solutions exactly.

\textbf{Continuous approximations (parallelogram tilings).} For continuous approximations based on parallelogram tilings (as constructed in Lemma~\ref{lem:area_density}), the deviations must be equal at all vertices: add constraint $d_1 = d_2 = d_3$ (see Remark~\ref{rem:continuous-constraint} for the derivation of this constraint).

\textbf{Interpolation.} For interpolation, the approximation must match $F$ at vertices: add constraint $d_1 = d_2 = d_3 = 0$.

\begin{remark}[Normalization to $\varepsilon = 1$]
\label{rem:normalization}
The optimization models can be normalized by setting $\varepsilon = 1$. The solution for general $\varepsilon$ is recovered via scaling: coordinates scale as $(x, y) = \sqrt{\varepsilon} \cdot (\bar{x}, \bar{y})$ and deviations as $d_i = \varepsilon \cdot \bar{d}_i$, where barred quantities denote the normalized solution. Area then scales as $A = \varepsilon \cdot \bar{A}$ and density as $\delta = \bar{\delta}/\varepsilon$.\footnote{Code to verify these models is available at \url{https://github.com/RobertHildebrand/OptimalTriangulationsForQuadratics}.}
\end{remark}

\subsection{Reformulation via slack variables and edge-product bounds}
\label{sec:slack-reformulation}

We now reformulate model~\eqref{model:optimal-triangle} in terms of edge products and slack variables, which eliminates the vertex coordinates entirely and reveals the structure needed for a rigorous solution.

\subsubsection{Symmetric bounds and sign reduction}

Let $\sigma := E_{\max} - E_{\min}$ denote the width of the error interval. For general approximations, $\sigma = 2\varepsilon$; for under/overestimation, $\sigma = \varepsilon$.
The following lemma shows that we may center the error interval at zero, which simplifies the subsequent analysis by making the bounds symmetric.
\begin{lemma}[Symmetric bounds]
\label{lem:symmetric-bounds}
Without loss of generality, $E_{\min} = -\sigma/2$ and $E_{\max} = \sigma/2$.
\end{lemma}

\begin{proof}
Let $\bar{E} := (E_{\min} + E_{\max})/2$ be the midpoint of the error interval, and consider the shifted linear function $\ell' := \ell - \bar{E}$, which has deviations $d_i' := d_i - \bar{E}$.

We verify that this shift transforms model~\eqref{model:optimal-triangle} with bounds $[E_{\min}, E_{\max}]$ into the same model with the symmetric bounds $[-\sigma/2, \sigma/2]$, without changing the area or feasibility.

\emph{Box constraints.} The constraint $d_i \in [E_{\min}, E_{\max}]$ becomes $d_i' \in [E_{\min} - \bar{E},\, E_{\max} - \bar{E}] = [-\sigma/2,\, \sigma/2]$.

\emph{Edge constraints.} The error along any edge (Lemma~\ref{lem:max_error}\ref{lem:max_error:formula}) is $E(\lambda) = (1-\lambda)d_i + \lambda d_j + \lambda(1-\lambda)k$. Replacing $d_i, d_j$ by $d_i' + \bar{E}, d_j' + \bar{E}$:
\[
E(\lambda) = (1-\lambda)(d_i' + \bar{E}) + \lambda(d_j' + \bar{E}) + \lambda(1-\lambda)k = E'(\lambda) + \bar{E},
\]
where $E'(\lambda) := (1-\lambda)d_i' + \lambda d_j' + \lambda(1-\lambda)k$ is the error function for the shifted deviations. In particular, the extremal error transforms as $E^*(k, d_i, d_j) = E^*(k, d_i', d_j') + \bar{E}$. Therefore, the edge constraints $E_{\min} \leq E^* \leq E_{\max}$ in model~\eqref{model:optimal-triangle} become $-\sigma/2 \leq E^*(k, d_i', d_j') \leq \sigma/2$.

\emph{Objective and geometry.} Since $\ell' = \ell - \bar{E}$ differs from $\ell$ by a constant, the vertex positions $(x_2, y_2, x_3, y_3)$, the edge products $k_i = \Delta x \cdot \Delta y$, and the triangle area $A = \tfrac{1}{2}|x_3 y_2 - x_2 y_3|$ are all unchanged.

Thus the shifted problem has the same feasible triangles and the same objective, with bounds centered at zero.
\end{proof}

The following lemma allows us to fix the signs of two edge products. Recall from Section~\ref{sec:error-analysis} that for a triangle with one vertex at the origin and the other two at $v_2 = (x_2, y_2)$ and $v_3 = (x_3, y_3)$, the three edge products are
\[
\begin{aligned}
k_1 &:= x_2 y_2 
&&\quad \text{(edge $e_1$: origin to $v_2$)} \\
k_2 &:= x_3 y_3 
&&\quad \text{(edge $e_2$: origin to $v_3$)} \\
k_3 &:= (x_3 - x_2)(y_3 - y_2) 
&&\quad \text{(edge $e_3$: $v_2$ to $v_3$)}
\end{aligned}
\]
These satisfy the identity $k_1 + k_2 - k_3 = x_3 y_2 + x_2 y_3$ (see Lemma~\ref{lem:equivalence}(i) below). Since each edge product $k = \Delta x \cdot \Delta y$ depends only on the displacement vector along the edge, translating the triangle (Lemma~\ref{lem:invariance_trans}) preserves all three edge products. Relabeling which vertex is at the origin permutes the $k_i$ but does not change the multiset $\{k_1, k_2, k_3\}$.

\begin{lemma}[Sign reduction]
\label{lem:sign-reduction}
Without loss of generality, $k_1, k_2 \geq 0$.
\end{lemma}

\begin{proof}
With symmetric bounds, the function $F(x,-y) = -F(x,y)$ implies that reflecting $y \mapsto -y$ and defining $\ell'(x,y) := -\ell(x,-y)$ gives a feasible solution with deviations $-d_i$ (still in $[-\sigma/2, \sigma/2]$) and all edge products negated. This reflection, combined with vertex relabeling (Lemma~\ref{lem:invariance_trans}) to place the vertex opposite the negative-product edge at the origin, ensures $k_1, k_2 \geq 0$.

More precisely: if at most one edge product is non-negative, apply the reflection to negate all three; now at least two are non-negative. Then relabel so that the two non-negative products correspond to the edges from the origin.
\end{proof}

\subsubsection{Slack variables and edge-product bounds}

Define the \emph{slack variables}
\begin{equation}
\label{eq:slack-def}
a_i := \sqrt{E_{\max} - d_i}, \qquad b_i := \sqrt{d_i - E_{\min}}, \qquad i = 1,2,3,
\end{equation}
satisfying $a_i^2 + b_i^2 = \sigma$ for each~$i$ and $a_i, b_i \in [0, \sqrt{\sigma}]$. Intuitively, $a_i$ measures how far the deviation $d_i$ is from the upper error cap $E_{\max}$, while $b_i$ measures its distance from the lower cap $E_{\min}$.

\begin{lemma}[Edge-product bounds]
\label{lem:edge-bounds}
Consider an edge with endpoint deviations $d_i, d_j \in [E_{\min}, E_{\max}]$ and edge product $k := \Delta x \cdot \Delta y$, where $(\Delta x, \Delta y)$ is the displacement vector along the edge (see Section~\ref{sec:error-analysis}). The error along the edge is $E(\lambda) = (1-\lambda)d_i + \lambda d_j + \lambda(1-\lambda)k$ by Lemma~\ref{lem:max_error}\ref{lem:max_error:formula}. Then the constraint $E_{\min} \leq E(\lambda) \leq E_{\max}$ for all $\lambda \in [0,1]$ is equivalent to:
\begin{enumerate}
\item[\emph{(a)}] If $k > 0$: \; $k \leq (a_i + a_j)^2$, \; where $a_i = \sqrt{E_{\max} - d_i}$.
\item[\emph{(b)}] If $k < 0$ (write $w = -k$): \; $w \leq (b_i + b_j)^2$, \; where $b_i = \sqrt{d_i - E_{\min}}$.
\item[\emph{(c)}] If $k = 0$: the constraint holds for all $d_i, d_j \in [E_{\min}, E_{\max}]$ with no further restriction on the geometry, since $E(\lambda) = (1-\lambda)d_i + \lambda d_j$ is a convex combination of $d_i$ and $d_j$, hence lies in $[E_{\min}, E_{\max}]$.
\end{enumerate}
\end{lemma}

\begin{proof}
Since $E(0) = d_i$ and $E(1) = d_j$ are both in $[E_{\min}, E_{\max}]$, only interior extrema need checking.

\emph{Part~(c):} When $k = 0$, $E(\lambda)$ is linear (affine) in $\lambda$, so it has no interior extremum. Its values on $[0,1]$ lie between $E(0) = d_i$ and $E(1) = d_j$, both of which are in $[E_{\min}, E_{\max}]$.

\emph{Part~(a):} When $k > 0$, $E(\lambda)$ is concave in $\lambda$ (the coefficient of $\lambda^2$ is $-k < 0$), so $\min_\lambda E = \min(d_i, d_j) \geq E_{\min}$ automatically. The upper bound $\max E \leq E_{\max}$ requires controlling the interior maximum. Substituting $\lambda^*$ from~\eqref{eq:lambda-star} into $E^* \leq E_{\max}$ and simplifying yields the quadratic inequality $k^2 + 2(d_i + d_j - 2E_{\max})k + (d_i - d_j)^2 \leq 0$ in $k$. The roots of the left-hand side are $(\sqrt{E_{\max} - d_i} \pm \sqrt{E_{\max} - d_j})^2 = (a_i \pm a_j)^2$. Since $k > 0$, the binding root is the larger one: $k \leq (a_i + a_j)^2$.

\emph{Part~(b):} When $k < 0$, $E(\lambda)$ is convex (the coefficient of $\lambda^2$ is $-k > 0$), so $\max_\lambda E = \max(d_i, d_j) \leq E_{\max}$ automatically. The lower bound $\min E \geq E_{\min}$ yields, by an analogous quadratic analysis, $-k \leq (b_i + b_j)^2$, i.e., $w \leq (b_i + b_j)^2$.
\end{proof}

\subsubsection{Reduced formulation and equivalence}

We now combine the edge-product bounds with an area identity to reformulate model~\eqref{model:optimal-triangle} as an optimization over $(a_i, k_i)$.

Recall that $b_i = \sqrt{d_i - E_{\min}} = \sqrt{\sigma - a_i^2}$ is determined by $a_i$, so in what follows the only free deviation variables are $a_1, a_2, a_3 \in [0, \sqrt{\sigma}]$.

Applying Lemma~\ref{lem:edge-bounds} to each of the three edges (using $k_1, k_2 \geq 0$ from Lemma~\ref{lem:sign-reduction}), the nonconvex edge-error constraints decouple into independent interval constraints for each $k_i$. This is the key structural step: it replaces the coupled quadratic constraints of model~\eqref{model:optimal-triangle} with simple box constraints. Define the \emph{box endpoints}:
\begin{equation}
\label{eq:box-endpoints}
P_{12} := (a_1 + a_2)^2, \quad P_{13} := (a_1 + a_3)^2, \quad Q := (b_2 + b_3)^2, \quad P_{23} := (a_2 + a_3)^2,
\end{equation}
where $b_i = \sqrt{\sigma - a_i^2}$. The $P_{ij}$ are upper bounds from the $a$-slacks (distance to $E_{\max}$), controlling the ascending edges; $Q$ is the upper bound from the $b$-slacks (distance to $E_{\min}$), controlling the descending edge. The feasible ranges for the edge products are:
\begin{equation}
\label{eq:box}
k_1 \in [0, P_{12}], \qquad k_2 \in [0, P_{13}], \qquad k_3 \in [-Q, P_{23}].
\end{equation}
The lower bound $0$ for $k_1, k_2$ comes from the sign reduction; the upper bounds from Lemma~\ref{lem:edge-bounds}(a). For $k_3$, the upper bound $P_{23}$ applies when $k_3 > 0$ (part~(a)), and the lower bound $-Q$ applies when $k_3 < 0$ (part~(b), with $w = -k_3$). When $k_3 = 0$, the constraint is vacuous (part~(c)).

\begin{lemma}[Equivalence of the reduced formulation]
\label{lem:equivalence}
\leavevmode
\begin{enumerate}
\item[\emph{(i)}] \textbf{Area identity.} For any triangle with $v_1 = (0,0)$ and edge products $k_1, k_2, k_3$:
\begin{equation}
\label{eq:area-identity}
4A^2 = (k_1 + k_2 - k_3)^2 - 4k_1 k_2.
\end{equation}
\item[\emph{(ii)}] \textbf{Realizability.} For any $k_1, k_2 > 0$ and $k_3 \in \R$ with $(k_1 + k_2 - k_3)^2 - 4k_1 k_2 > 0$, there exists a non-degenerate triangle $v_1 = (0,0), v_2, v_3$ with the prescribed edge products.
\item[\emph{(iii)}] \textbf{Equivalence.} The following \emph{reduced problem} has the same optimal value as model~\eqref{model:optimal-triangle}:
\begin{equation}
\label{eq:reduced}
\max_{\substack{a_i \in [0,\sqrt{\sigma}]\\[1pt] k_i \in \text{box}~\eqref{eq:box}}}
\; 4A^2 = (k_1 + k_2 - k_3)^2 - 4k_1 k_2.
\end{equation}
Moreover, any optimizer of~\eqref{eq:reduced} with $k_1, k_2 > 0$ can be realized as a triangle that is optimal for model~\eqref{model:optimal-triangle}.
\end{enumerate}
\end{lemma}

\begin{proof}
\emph{Part~(i).} Expanding: $k_1 + k_2 - k_3 = x_2 y_2 + x_3 y_3 - (x_3 - x_2)(y_3 - y_2) = x_3 y_2 + x_2 y_3$. Therefore $(k_1 + k_2 - k_3)^2 - 4k_1 k_2 = (x_3 y_2 + x_2 y_3)^2 - 4x_2 y_2 x_3 y_3 = (x_3 y_2 - x_2 y_3)^2 = 4A^2$.

\emph{Part~(ii).} We construct $v_2$ and $v_3$ explicitly. For any free parameter $t > 0$, set $v_2 = (t,\, k_1/t)$, ensuring $x_2 y_2 = k_1$. Set $v_3 = (s,\, k_2/s)$ for $s \neq 0$ to be determined, ensuring $x_3 y_3 = k_2$. The third edge product is
\[
(x_3 - x_2)(y_3 - y_2) = (s - t)\Bigl(\frac{k_2}{s} - \frac{k_1}{t}\Bigr) = k_1 + k_2 - \frac{sk_1}{t} - \frac{tk_2}{s}.
\]
Setting this equal to $k_3$ and multiplying through by $st$ yields the quadratic in~$s$:
\begin{equation}
\label{eq:realizable-quad}
k_1 s^2 - (k_1 + k_2 - k_3)\,t\,s + k_2 t^2 = 0.
\end{equation}
Since $k_1 > 0$, this is a genuine quadratic with discriminant $\Delta = [(k_1 + k_2 - k_3)^2 - 4k_1 k_2]\,t^2 = 4A^2\,t^2 > 0$, giving two distinct real roots. By Vieta's formulas their product is $k_2 t^2/k_1 > 0$, so both roots are nonzero. Taking either root gives a non-degenerate triangle with the prescribed edge products and area $A > 0$.

\emph{Part~(iii).} We show both directions.

\emph{Model~\eqref{model:optimal-triangle} $\to$ problem~\eqref{eq:reduced}:} Given a feasible triangle for model~\eqref{model:optimal-triangle}, apply Lemma~\ref{lem:symmetric-bounds} (symmetric bounds), Lemma~\ref{lem:sign-reduction} (sign reduction), and Lemma~\ref{lem:invariance_trans} (translation to place one vertex at the origin). Compute $a_i = \sqrt{E_{\max} - d_i}$ and read off the edge products $k_1, k_2, k_3$. By Lemma~\ref{lem:edge-bounds}, the error constraints of model~\eqref{model:optimal-triangle} imply the box constraints~\eqref{eq:box}. By part~(i), the area objective equals $(k_1 + k_2 - k_3)^2 - 4k_1 k_2$. So the feasible set of~\eqref{eq:reduced} contains the image of every feasible triangle, with the same objective value.

\emph{Problem~\eqref{eq:reduced} $\to$ model~\eqref{model:optimal-triangle}:} Given an optimizer $(a_i^*, k_i^*)$ of~\eqref{eq:reduced} with $k_1, k_2 > 0$, part~(ii) constructs a non-degenerate triangle with the prescribed edge products. Setting $d_i = E_{\max} - a_i^{*2}$ recovers deviations in $[E_{\min}, E_{\max}]$. The box constraints~\eqref{eq:box}, via Lemma~\ref{lem:edge-bounds}, guarantee the error constraints of model~\eqref{model:optimal-triangle} are satisfied. By part~(i), the triangle has the same area as the objective value of~\eqref{eq:reduced}.

Thus the geometry of feasible triangles is encoded exactly by the box~\eqref{eq:box} in edge-product space together with the quadratic area formula~\eqref{eq:area-identity}.
\end{proof}

\subsection{Solving the optimization problem}
\label{sec:elem-proof}

We now solve problem~\eqref{eq:reduced} through a sequence of exact reductions: eliminating the edge-product variables via separate convexity, establishing dominance of a particular box vertex, proving symmetry of the optimal slack variables, and certifying the global optimum algebraically.

\subsubsection{Eliminating the edge-product variables}

The first reduction eliminates the continuous optimization over the edge products $k_i$: only finitely many sign-pattern choices remain.

\begin{lemma}[Vertex reduction]
\label{lem:vertex-reduction}
For fixed $a_i$, the maximum of $4A^2$ over the box~\eqref{eq:box} is attained at a vertex of the box.
\end{lemma}

\begin{proof}
$4A^2$ is quadratic in each $k_i$ with leading coefficient $+1$, hence separately convex. A separately convex function on a product of intervals is maximized at a vertex of the box.
\end{proof}

For each fixed $(a_1, a_2, a_3)$, the box~\eqref{eq:box} has eight vertices obtained by setting each $k_i$ to one of its two endpoints. Since interchanging a maximum over a finite set with a supremum over a compact set is valid, it suffices to compare, for each vertex type, the best value obtainable over all deviations. 
We label the eight vertices by their sign pattern. 
Of particular importance is the vertex
$
(k_1, k_2, k_3) = (P_{12}, P_{13}, -Q),
$
which we call \emph{Type~A}. This configuration saturates all edge-product bounds: the two ascending edges are maximized, while the descending edge is minimized. The remaining types (B through E) each set one or more $k_i$ to the opposite endpoint. The following lemma shows that Type~A yields the largest objective value among all box vertices.

\begin{lemma}[Type~A dominance]
\label{lem:tight}
Over all $a_i \in [0, \sqrt{\sigma}]$, the largest value of $4A^2$ among the eight box-vertex objectives is attained by the Type~A vertex $(k_1, k_2, k_3) = (P_{12}, P_{13}, -Q)$.
\end{lemma}

\begin{proof}
We exhibit a lower bound for Type~A and upper bounds for each competitor.

\smallskip
\emph{Type~A} $(P_{12}, P_{13}, -Q)$: Taking $a_1 = \sqrt{\sigma}$, $a_2 = a_3 = \sqrt{\sigma/9}$ gives $P_{12} = P_{13} = \tfrac{16\sigma}{9}$, $Q = \tfrac{32\sigma}{9}$, and $4A^2 = \tfrac{1024\sigma^2}{27}$, so
\[
\max_{a_i} A^2\big|_{\text{Type A}} \;\geq\; \tfrac{256}{27}\,\sigma^2.
\]

\emph{Type~B} $(P_{12}, P_{13}, P_{23})$: Since
\[
P_{12} + P_{13} - P_{23} = (a_1{+}a_2)^2 + (a_1{+}a_3)^2 - (a_2{+}a_3)^2 = 2a_1(a_1{+}a_2{+}a_3) - 2a_2 a_3,
\]
a direct expansion gives
\[
4A^2 = (P_{12}{+}P_{13}{-}P_{23})^2 - 4P_{12}P_{13} = -16\,a_1 a_2 a_3(a_1{+}a_2{+}a_3) \leq 0.
\]

\emph{Types~C} ($k_1 = 0$ or $k_2 = 0$, $k_3 = -Q$): Take $k_1 = 0$, $k_2 = P_{13}$, $k_3 = -Q$. Then $4A^2 = (P_{13} + Q)^2$, so it suffices to maximize $P_{13} + Q = (a_1{+}a_3)^2 + (b_2{+}b_3)^2$. Since $a_1 \leq \sqrt{\sigma}$ and $b_2 \leq \sqrt{\sigma}$:
\[
(a_1{+}a_3)^2 + (b_2{+}b_3)^2 \;\leq\; (\sqrt{\sigma}{+}a_3)^2 + (\sqrt{\sigma}{+}b_3)^2 \;=\; 3\sigma + 2\sqrt{\sigma}(a_3{+}b_3).
\]
Since $(a_3 + b_3)^2 \leq 2(a_3^2 + b_3^2) = 2\sigma$ (by Cauchy--Schwarz), we have $a_3 + b_3 \leq \sqrt{2\sigma}$. Therefore $P_{13} + Q \leq \sigma(3 + 2\sqrt{2})$, giving
\[
A^2 \;\leq\; \tfrac{(3+2\sqrt{2})^2}{4}\,\sigma^2 \;=\; \tfrac{17 + 12\sqrt{2}}{4}\,\sigma^2.
\]

\emph{Types~D} ($k_1 = 0$ or $k_2 = 0$, $k_3 = P_{23}$): Take $k_1 = 0$, $k_2 = P_{13}$, $k_3 = P_{23}$. Then
\[
4A^2 = (P_{13} - P_{23})^2 = \bigl[(a_1{-}a_2)(a_1{+}2a_3{+}a_2)\bigr]^2.
\]
Since $|a_1 - a_2| \leq \sqrt{\sigma}$ and $a_1 + 2a_3 + a_2 \leq 4\sqrt{\sigma}$: $4A^2 \leq 16\sigma^2$, so $A^2 \leq 4\sigma^2$.

\emph{Types~E} ($k_1 = k_2 = 0$): $4A^2 = k_3^2$. Since $k_3 \in [-Q, P_{23}]$ and both $Q, P_{23} \leq 4\sigma$ (by Cauchy--Schwarz on $a_i^2 + b_i^2 = \sigma$), $4A^2 \leq 16\sigma^2$, so $A^2 \leq 4\sigma^2$.

\smallskip
Since $\tfrac{256}{27} > \tfrac{17+12\sqrt{2}}{4} > 4$ (numerically $\approx 9.48 > 8.49 > 4$), Type~A dominates.
\end{proof}
\begin{remark}[Geometric interpretation and relation to prior work]
At the Type~A optimum, all edge-product bounds are saturated. By Lemma~\ref{lem:edge-bounds}, these bounds arise from the interior extremum constraints of the edge error functions. Thus, the optimal configuration is characterized by the fact that the extremal error along each edge reaches one of the prescribed bounds at its interior critical point (while endpoint errors may also be active).

This provides an optimization-based counterpart to the geometric arguments used in Atariah et al. \cite{Atariah:2018}, where it is shown that if the interior extremum is not active, one can deform the triangle to increase its area while maintaining feasibility. In the present formulation, this behavior is captured algebraically: increasing an edge product corresponds to stretching the edge, and the bounds $(P_{ij}, Q)$ mark exactly the point where the interior extremum constraint becomes tight. The global optimum is therefore attained when all such constraints are active simultaneously.
\end{remark}
\begin{corollary}
\label{cor:reduce-to-a}
The optimum of~\eqref{eq:reduced} is attained by a Type~A vertex and is realized by an actual triangle. The Type~A witness $a_1 = \sqrt{\sigma}$, $a_2 = a_3 = \sqrt{\sigma/9}$ gives $k_1^* = k_2^* = \tfrac{16\sigma}{9} > 0$ and $4A^{*2} = \tfrac{1024\sigma^2}{27} > 0$, so Lemma~\ref{lem:equivalence}(ii) confirms realizability.
\end{corollary}

By Corollary~\ref{cor:reduce-to-a}, the optimum of~\eqref{eq:reduced} is attained at the Type~A vertex $(k_1, k_2, k_3) = (P_{12}, P_{13}, -Q)$. Substituting these into the area identity~\eqref{eq:area-identity} eliminates the edge-product variables $k_i$, reducing~\eqref{eq:reduced} to an optimization in the three slack variables alone:
\begin{equation}
\label{eq:3var}
\max_{a_1, a_2, a_3 \in [0, \sqrt{\sigma}]} 4A^2 = (P_{12} - P_{13})^2 + Q\,(Q + 2P_{12} + 2P_{13}).
\end{equation}

\subsubsection{Symmetry: $a_2 = a_3$}

\begin{theorem}
\label{thm:symmetry}
The maximum of~\eqref{eq:3var} is attained at $a_2 = a_3$.
\end{theorem}

The intuition is geometric: after the preceding reductions, edges $e_1$ and $e_2$ play symmetric roles (both ascending, both from the origin), so any asymmetry between $a_2$ and $a_3$ should be improvable. The proof makes this precise by showing that $4A^2$ is a strictly decreasing function of the asymmetry $h := (a_2 - a_3)^2/4$.

\begin{proof}
If $\bar{a} := (a_2 + a_3)/2 = 0$ then $a_2 = a_3 = 0$ and the conclusion holds trivially. Assume $\bar{a} > 0$ for the remainder. We fix $a_1 \geq 0$ and show that among all $(a_2, a_3)$ with a given mean $\bar{a} \in (0, \sqrt{\sigma})$, the symmetric choice $a_2 = a_3 = \bar{a}$ is the unique maximizer.

\emph{Setup.}
Write $a_2 = \bar{a} + H$, $a_3 = \bar{a} - H$, so $H = 0$ is the symmetric case. By the $a_2 \leftrightarrow a_3$ symmetry of~\eqref{eq:3var}, we may assume $H \geq 0$. Using the abbreviations
\[
C := (a_1 + \bar{a})^2, \qquad S := \sigma - \bar{a}^2, \qquad h := H^2,
\]
we claim that $Q := (b_2 + b_3)^2$ and the objective $4A^2$ both depend on $H$ only through $h = H^2$, so we may write $Q(h)$ and $4A^2(h)$.

For $Q$: since $b_i^2 = \sigma - a_i^2$, we have $b_2^2 = S - 2\bar{a}H - h$ and $b_3^2 = S + 2\bar{a}H - h$. While $b_2$ and $b_3$ individually depend on $H$, their symmetric functions do not: $b_2^2 + b_3^2 = 2(S-h)$ and $b_2^2 b_3^2 = (S-h)^2 - 4\bar{a}^2 h$, both functions of $h$ alone. Since $Q = b_2^2 + b_3^2 + 2b_2 b_3$, we get
\begin{equation}
\label{eq:Q-of-h}
Q(h) = 2(S-h) + 2\sqrt{(S-h)^2 - 4\bar{a}^2 h},
\end{equation}
a function of $h$ alone.

For $4A^2$: from~\eqref{eq:3var}, $4A^2 = (P_{12} - P_{13})^2 + Q(Q + 2P_{12} + 2P_{13})$. Now $P_{12} - P_{13} = (a_1{+}\bar{a}{+}H)^2 - (a_1{+}\bar{a}{-}H)^2 = 4H(a_1+\bar{a})$, so $(P_{12} - P_{13})^2 = 16hC$. Also $P_{12} + P_{13} = 2C + 2h$. Therefore
\begin{equation}
\label{eq:4A2-h}
4A^2(h) = 16C\,h + Q(h)\bigl[Q(h) + 4C + 4h\bigr],
\end{equation}
a function of $h$ alone (the signed $H$ cancels upon squaring). At $h = 0$: $Q(0) = 4S$ and $4A^2(0) = 16S(S + C)$.

\emph{Key bound.}
The entire proof rests on one strict inequality: for $h > 0$,
\begin{equation}
\label{eq:Q-strict}
Q(h) < 4(S - h).
\end{equation}
From~\eqref{eq:Q-of-h}, $Q = 2(S{-}h) + 2\sqrt{(S{-}h)^2 - 4\bar{a}^2 h}$. Since $S - h > 0$ (because $a_2 \leq \sqrt{\sigma}$ forces $h \leq (\sqrt{\sigma}-\bar{a})^2 < S$) and $4\bar{a}^2 h > 0$, we have $\sqrt{(S{-}h)^2 - 4\bar{a}^2 h} < S - h$, giving $Q < 2(S{-}h) + 2(S{-}h) = 4(S{-}h)$.

\emph{Conclusion.}
Since $Q > 0$ and $4C + 4h \geq 0$, the map $x \mapsto x(x + 4C + 4h)$ is increasing for $x > 0$. Applying~\eqref{eq:Q-strict}:
\[
Q(Q + 4C + 4h) < 4(S{-}h)\bigl(4(S{-}h) + 4C + 4h\bigr) = 16(S-h)(S + C).
\]
Substituting into~\eqref{eq:4A2-h}:
\[
4A^2(h) < 16Ch + 16(S-h)(S+C) = 16S(S+C) - 16Sh = 4A^2(0) - 16Sh.
\]
Since $S > 0$ and $h > 0$, we conclude $4A^2(h) < 4A^2(0)$.
\end{proof}

\begin{remark}[Strict concavity]
\label{rem:concavity}
The proof above shows $4A^2(h) < 4A^2(0)$ for $h > 0$ via an elementary monotonicity argument. Alternatively, one can verify that $4A^2$ is strictly concave in~$h$: differentiating~\eqref{eq:4A2-h} twice and using the identity $(S - h + 2\bar{a}^2)^2 - (b_2 b_3)^2 = 4\bar{a}^2\sigma$ yields $(4A^2)'' = -32\bar{a}^2\sigma(\sigma + a_1^2 + 2a_1\bar{a})/(b_2 b_3)^3 < 0$.
\end{remark}

\subsubsection{Optimality certificate}

With $a_2 = a_3$ established, we write $a := a_1$ and $b := a_2 = a_3$. Problem~\eqref{eq:3var} reduces to a two-variable optimization:
\begin{equation}
\label{eq:2var-final}
\max_{a, b \in [0, \sqrt{\sigma}]} 4A^2 = 16(\sigma - b^2)(\sigma + a^2 + 2ab).
\end{equation}
We solve this by exhibiting an \emph{optimality certificate}: a nonnegative decomposition of the gap between the claimed maximum and the objective, which simultaneously proves the bound and identifies when equality holds.

\begin{theorem}[Optimality certificate]
\label{thm:certificate}
The candidate optimum is $4A^{*2} = \tfrac{1024\sigma^2}{27}$, achieved at $a^* = \sqrt{\sigma}$, $b^* = \sqrt{\sigma}/3$. The gap $\tfrac{1024\sigma^2}{27} - 4A^2$ admits the following nonnegative decomposition for all $a, b \in [0, \sqrt{\sigma}]$:
\begin{equation}
\label{eq:certificate}
\frac{1024\sigma^2}{27} - 4A^2 = \frac{32\sqrt{\sigma}}{27}(3b - \sqrt{\sigma})^2(3b + 5\sqrt{\sigma}) + 16(\sigma - b^2)(\sqrt{\sigma} - a)(\sqrt{\sigma} + a + 2b).
\end{equation}
Both terms on the right are nonnegative on $[0, \sqrt{\sigma}]^2$. Equality holds if and only if $a = \sqrt{\sigma}$ and $b = \sqrt{\sigma}/3$.
\end{theorem}

\begin{proof}
Let $f(a, b) := 16(\sigma - b^2)(\sigma + a^2 + 2ab)$. At $a = \sqrt{\sigma}$:
\[
f|_{a = \sqrt{\sigma}} = 32\sqrt{\sigma}\,(\sqrt{\sigma} - b)(\sqrt{\sigma} + b)^2,
\]
and using the factorization $4C^3 - 27(C-t)t^2 = (3t - 2C)^2(3t + C)$ with $C = 2\sqrt{\sigma}$, $t = \sqrt{\sigma} + b$:
\begin{equation}
\label{eq:1d-cert}
\tfrac{1024\sigma^2}{27} - f|_{a = \sqrt{\sigma}} = \tfrac{32\sqrt{\sigma}}{27}(3b - \sqrt{\sigma})^2(3b + 5\sqrt{\sigma}).
\end{equation}
For general~$a$:
\begin{equation}
\label{eq:a-gap}
f|_{a = \sqrt{\sigma}} - f = 16(\sigma - b^2)(\sqrt{\sigma} - a)(\sqrt{\sigma} + a + 2b).
\end{equation}
Adding~\eqref{eq:1d-cert} and~\eqref{eq:a-gap} gives~\eqref{eq:certificate}. The first summand is nonneg because $(3b - \sqrt{\sigma})^2 \geq 0$ and $3b + 5\sqrt{\sigma} > 0$; the second because each factor $\sigma - b^2$, $\sqrt{\sigma} - a$, and $\sqrt{\sigma} + a + 2b$ is nonneg on $[0, \sqrt{\sigma}]^2$.

For equality, the second summand must vanish, giving $a = \sqrt{\sigma}$ (the factor $\sigma - b^2$ is positive except at $b = \sqrt{\sigma}$, which makes $4A^2 = 0$). Then the first summand gives $b = \sqrt{\sigma}/3$.
\end{proof}

\subsubsection{Combined certificate for the full problem}

\begin{theorem}[Combined certificate]
\label{thm:combined}
The gap $G := \frac{1024\sigma^2}{27} - 4A^2$ satisfies $G \geq 0$ on the full feasible set of~\eqref{eq:reduced}.
\end{theorem}

\begin{proof}
The proof has two parts.

\emph{Part 1: Reduction to box vertices.}
Since $4A^2$ is quadratic in each $k_i$ with leading coefficient $+1$, the gap $G$ is quadratic in each $k_i$ with leading coefficient $-1$, hence concave in each $k_i$ separately. A concave function on an interval attains its minimum at an endpoint. Applying this to each $k_i$ in succession: $\min_{\text{box}} G = \min_{\text{8 vertices}} G$. It therefore suffices to show $G(v) \geq 0$ at each of the eight box vertices.

\emph{Part 2: Nonnegativity at each vertex.}
\emph{Type~A} $(P_{12}, P_{13}, -Q)$: By Theorem~\ref{thm:symmetry}, $4A^2(a_1, a_2, a_3)|_{\text{Type A}} \leq 4A^2(a_1, \bar{a}, \bar{a})|_{\text{Type A}} = 16(\sigma - \bar{a}^2)(\sigma + a_1^2 + 2a_1\bar{a})$. Theorem~\ref{thm:certificate} gives $G \geq 0$.

\emph{Type~B}: $4A^2 \leq 0$, so $G \geq \frac{1024\sigma^2}{27}$.

\emph{Types~C}: $A^2 \leq \tfrac{17 + 12\sqrt{2}}{4}\sigma^2$, so $G \geq \tfrac{1024 - 27(17 + 12\sqrt{2})}{27}\sigma^2 > 0$ since $27(17 + 12\sqrt{2}) \approx 917.2 < 1024$.

\emph{Types~D, E}: $A^2 \leq 4\sigma^2$, so $G \geq \tfrac{592\sigma^2}{27}$.
\end{proof}

\subsubsection{Optimal parameters}

\begin{corollary}[Optimal triangle]
\label{cor:optimal}
With $\sigma = E_{\max} - E_{\min}$, the maximum area is
\[
A^* = \frac{16\sqrt{3}}{9}\sigma = \frac{16\sqrt{3}}{9}(E_{\max} - E_{\min}),
\]
achieved at the slack variables $a^* = \sqrt{\sigma}$ and $b^* = \sqrt{\sigma}/3$, corresponding to:
\begin{align*}
k^* = k_1^* = k_2^* &= \frac{16\sigma}{9}, &
w^* = -k_3^* &= \frac{32\sigma}{9}, &
d_1^* &= E_{\min}, &
d_2^* = d_3^* &= \frac{E_{\min} + 8E_{\max}}{9}.
\end{align*}
\end{corollary}

\begin{proof}
From the certificate (Theorem~\ref{thm:certificate}), the optimal slack variables are $a^* = a_1^* = \sqrt{\sigma}$ and $b^* = a_2^* = a_3^* = \sqrt{\sigma}/3$.
Recovering deviations via $d_i = E_{\max} - a_i^2$: $d_1^* = E_{\max} - \sigma = E_{\min}$ and $d_2^* = d_3^* = E_{\max} - \sigma/9 = (E_{\min} + 8E_{\max})/9$.
Edge products: $k^* = (a^* + b^*)^2 = (4\sqrt{\sigma}/3)^2 = 16\sigma/9$. For $w^*$: $b_2^{\text{slack}} = b_3^{\text{slack}} = \sqrt{\sigma - (b^*)^2} = \sqrt{8\sigma/9}$, so $w^* = (2b_2^{\text{slack}})^2 = 32\sigma/9$. The maximum squared area is $4A^{*2} = \tfrac{1024\sigma^2}{27}$, giving $A^* = \tfrac{16\sqrt{3}}{9}\sigma$.
\end{proof}

By Lemma~\ref{lem:equivalence}(ii), this optimum is realized by an actual triangle. The pE-normalization (Corollary~\ref{cor:pe-normalization}) shows the optimal triangle can be placed in the symmetric form $(x_2, y_2) = (y_3, x_3)$ with $x_2 > y_2 > 0$.

\begin{remark}[Geometric interpretation of the optimal triangle]
\label{rem:equilateral-pe}
The optimal triangles from Corollary~\ref{cor:optimal} have two notable geometric properties.

\emph{Pseudo-Euclidean equilateral property.} In the pseudo-Euclidean plane with metric $ds^2 = dx\,dy$, the ``pseudo-length'' of an edge is the product of its coordinate differences. The two ascending edges of the optimal triangle have equal pseudo-length $k^*$, and the descending edge has pseudo-length $-w^* = -2k^*$. This ``equilateral'' characterization was first noted by Pottmann et al.\ \cite{Pottmann:2000} for interpolations and extended by Atariah et al.\ \cite{Atariah:2018}.

\emph{Hyperbola touch property.} The non-origin vertices of the optimal triangle lie on the hyperbola $xy = k^*$, where $k^* = \tfrac{32\varepsilon}{9}$ for the general case and $k^* = \tfrac{8\varepsilon}{3}$ for the continuous case. This follows from $k_1 = k_2 = k^*$ at optimality: both ascending edge products equal the same value, so both vertices satisfy $x_i y_i = k^*$. This property is related to the observation of Atariah et al.\ that optimal triangles ``touch'' a hyperbola; the specific hyperbola depends on the approximation type.
\end{remark}

\subsection{Scaling to specific cases}
\label{sec:scaling}

Corollary~\ref{cor:optimal} immediately yields the optimal densities for all approximation types:
\begin{center}
\renewcommand{\arraystretch}{1.3}
\begin{tabular}{lcccc}
\toprule
Type & $E_{\min}$ & $E_{\max}$ & $\sigma = E_{\max} - E_{\min}$ & Optimal Area \\
\midrule
General & $-\varepsilon$ & $\varepsilon$ & $2\varepsilon$ & $\tfrac{32\sqrt{3}}{9}\varepsilon$ \\
Underestimation & $-\varepsilon$ & $0$ & $\varepsilon$ & $\tfrac{16\sqrt{3}}{9}\varepsilon$ \\
Overestimation & $0$ & $\varepsilon$ & $\varepsilon$ & $\tfrac{16\sqrt{3}}{9}\varepsilon$ \\
\bottomrule
\end{tabular}
\end{center}

\textbf{General approximations} ($\sigma = 2\varepsilon$): The optimal area is $A^* = \frac{32\sqrt{3}}{9}\varepsilon$, giving triangle density $\frac{3\sqrt{3}}{32\varepsilon}$.

The optimal deviations are $d_1^* = -\varepsilon$ and $d_2^* = d_3^* = \frac{-\varepsilon + 8\varepsilon}{9} = \frac{7\varepsilon}{9}$.

\textbf{Under/overestimations} ($\sigma = \varepsilon$): The optimal area is $A^* = \frac{16\sqrt{3}}{9}\varepsilon$, giving triangle density $\frac{3\sqrt{3}}{16\varepsilon}$. See Figure~\ref{fig:under-over-triangle} for an illustration.

For underestimation: $d_1^* = -\varepsilon$, $d_2^* = d_3^* = \frac{-\varepsilon + 0}{9} = -\frac{\varepsilon}{9}$.

For overestimation: $d_1^* = 0$, $d_2^* = d_3^* = \frac{0 + 8\varepsilon}{9} = \frac{8\varepsilon}{9}$.

\textbf{Triangle coordinates (general case).} For completeness, we give the explicit vertex coordinates of the optimal triangle in the pE-normalized symmetric form from Corollary~\ref{cor:pe-normalization}. From $w^* = \tfrac{64\varepsilon}{9}$, we get $(x_2-y_2)^2 = w^* = \tfrac{64\varepsilon}{9}$, so $x_2 - y_2 = \tfrac{8\sqrt{\varepsilon}}{3}$. Combined with $x_2 y_2 = k^* = \tfrac{32\varepsilon}{9}$:
\[
x_2 = \tfrac{4}{3}(1+\sqrt{3})\sqrt{\varepsilon}, \quad y_2 = \tfrac{4}{3}(\sqrt{3}-1)\sqrt{\varepsilon}, \quad (x_3, y_3) = (y_2, x_2).
\]

\begin{remark}[Geometric similarity of optimal triangles]
\label{rem:similarity}
The general and continuous optimal triangles are \emph{geometrically similar}---they have the same shape but different sizes (\cref{fig:optimal-triangle-comparison}). Both have aspect ratio:
\[
\frac{x_2}{y_2} = \frac{1+\sqrt{3}}{\sqrt{3}-1} = 2 + \sqrt{3} \approx 3.73.
\]
The general case triangle is scaled by factor $\tfrac{2}{\sqrt{3}} \approx 1.155$ relative to the continuous case, yielding area ratio $\tfrac{4}{3}$ (33\% larger). This explains why the improvement from continuous to general approximations is modest compared to the improvement from regular grids: the optimal \emph{shape} is already achieved by continuous approximations; relaxing continuity only allows a larger \emph{scale}.

The interpolation triangle has a different shape with aspect ratio $\tfrac{3+\sqrt{5}}{2} = \varphi + 1 \approx 2.62$, where $\varphi$ is the golden ratio. The under/overestimation triangles also have aspect ratio $2 + \sqrt{3} \approx 3.73$, the same as the general and continuous cases. Indeed, all non-interpolation optimal triangles satisfy $w/k = 2$ (where $w = (x_2-y_2)^2$ and $k = x_2 y_2$). Setting $r = x_2/y_2$ and solving $(r-1)^2/r = w/k = 2$ yields $r^2 - 4r + 1 = 0$, hence $r = 2+\sqrt{3}$. Only the interpolation case, where $\Delta = 0$ forces $\lambda = 1/2$ and gives $w/k = 1$ instead, leads to the different aspect ratio $(3+\sqrt{5})/2$.
\end{remark}

\begin{figure}[ht]
\centering
\begin{tikzpicture}[scale=1.4, >=latex]
  \draw[->] (-0.3,0) -- (4.2,0) node[right] {$x$};
  \draw[->] (0,-0.3) -- (0,4.2) node[above] {$y$};
  
  \draw[dashed, gray!50] (0,0) -- (3.8,3.8);

  \draw[very thick, fill=blue!20, fill opacity=0.7] (0,0) -- (3.64,0.98) -- (0.98,3.64) -- cycle;
  \node[blue!70!black, font=\small, anchor=west] at (3.7,0.98) {General};
  
  \draw[thick, fill=green!20, fill opacity=0.7] (0,0) -- (3.15,0.85) -- (0.85,3.15) -- cycle;
  \node[green!50!black, font=\small, anchor=east] at (0.4,1.8) {Continuous};
  
  \draw[thick, fill=orange!20, fill opacity=0.7] (0,0) -- (3.24,1.24) -- (1.24,3.24) -- cycle;
  \node[orange!70!black, font=\small] at (1.7,1.9) {Interpolation};
  
  \draw[thick, dashed, red!70] (0,0) -- (2.0,0) -- (2.0,2.0) -- cycle;
  \node[red!70, font=\small] at (2.4,0.3) {Regular};
  
  \filldraw (0,0) circle (2pt);
  
  \node[draw, fill=white, align=left, font=\footnotesize, anchor=north west] at (4.5,4.0) {
    \textbf{Areas (for $\varepsilon=1$):}\\[3pt]
    \textcolor{blue!70!black}{General: $6.16$}\\[1pt]
    \textcolor{orange!70!black}{Interpolation: $4.47$}\\[1pt]
    \textcolor{green!50!black}{Continuous: $4.62$}\\[1pt]
    \textcolor{red!70}{Regular grid: $4.00$}
  };
  
\end{tikzpicture}
\caption{Comparison of optimal triangles for different approximation types with error bound $\varepsilon$. The general and continuous optimal triangles (blue and green) are \emph{geometrically similar}---they have the same shape (aspect ratio $2+\sqrt{3} \approx 3.73$) but the general case is scaled by factor $2/\sqrt{3} \approx 1.15$, yielding 33\% larger area. The interpolation triangle (orange) has a different shape with aspect ratio $(3+\sqrt{5})/2 \approx 2.62$, related to the golden ratio. All optimal triangles significantly outperform the regular grid (red dashed).}
\label{fig:optimal-triangle-comparison}
\end{figure}
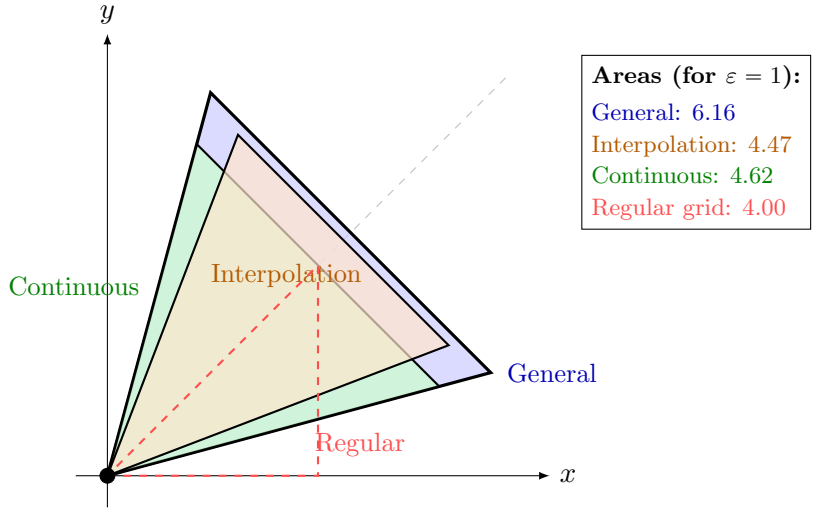

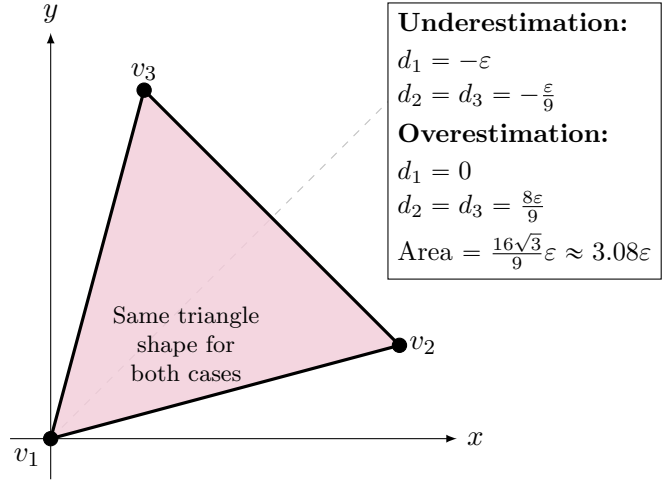
\begin{figure}[ht]
\centering
\begin{tikzpicture}[scale=1.8, >=latex]
  \draw[->] (-0.3,0) -- (3.0,0) node[right] {$x$};
  \draw[->] (0,-0.3) -- (0,3.0) node[above] {$y$};
  
  \draw[dashed, gray!50] (0,0) -- (2.6,2.6);
  
  \draw[very thick, fill=purple!20, fill opacity=0.8] (0,0) -- (2.576,0.690) -- (0.690,2.576) -- cycle;

  \filldraw (0,0) circle (1.5pt) node[below left] {$v_1$};
  \filldraw (2.576,0.690) circle (1.5pt) node[right] {$v_2$};
  \filldraw (0.690,2.576) circle (1.5pt) node[above] {$v_3$};
  
  \node[draw, fill=white, align=left, font=\small] at (3.5,2.2) {
    \textbf{Underestimation:}\\[2pt]
    $d_1 = -\varepsilon$\\[1pt]
    $d_2 = d_3 = -\tfrac{\varepsilon}{9}$\\[3pt]
    \textbf{Overestimation:}\\[2pt]
    $d_1 = 0$\\[1pt]
    $d_2 = d_3 = \tfrac{8\varepsilon}{9}$\\[3pt]
    Area $= \tfrac{16\sqrt{3}}{9}\varepsilon \approx 3.08\varepsilon$
  };
  
  \node[font=\footnotesize, align=center] at (1,0.7) {Same triangle\\shape for \\both cases};
  
\end{tikzpicture}
\caption{Optimal triangle for underestimation and overestimation. Both cases yield the same triangle geometry but with different vertex deviations. This triangle has aspect ratio $2+\sqrt{3} \approx 3.73$, identical to the general and continuous cases---the one-sided constraint changes the \emph{scale} (area is halved) but not the optimal \emph{shape}. For underestimation, all deviations are non-positive ($d_i \leq 0$); for overestimation, all are non-negative ($d_i \geq 0$). The optimal area is exactly half that of the general case.}
\label{fig:under-over-triangle}
\end{figure}

\section{Proofs of Main Theorems}
\label{sec:proofs}

We now assemble the results from the preceding sections to prove the main theorems. The key ingredients are: the reduction to a single triangle (Section~\ref{sec:reduction}), the edge-error formula (Section~\ref{sec:error-analysis}), and the slack-variable reformulation and optimality certificates (Sections~\ref{sec:slack-reformulation}--\ref{sec:elem-proof}).

\subsection{General and one-sided approximations}

\begin{proof}[Proof of \cref{thm:optimal-triangulation} (General \pwl approximations)]
\label{thm:optimal-triangulation-proof}
We prove that the optimal triangle density for $\varepsilon$-approximations of $F(x,y)=xy$ is $\tfrac{3\sqrt{3}}{32}\tfrac{1}{\varepsilon}$.

\textbf{Upper bound (achievability):} By Lemma~\ref{lem:area_density}, it suffices to construct a triangle $T$ with area $A = \tfrac{32\sqrt{3}}{9}\varepsilon$ that admits a linear $\varepsilon$-approximation. The triangle with vertices and deviations given in Section~\ref{sec:scaling} has area $\tfrac{32\sqrt{3}}{9}\varepsilon$ and satisfies all error constraints with equality $\pm\varepsilon$. By Lemma~\ref{lem:invariance_trans} and Lemma~\ref{lem:invariance_reflection}, we can tile the entire plane with copies and reflections of this triangle, achieving triangle density $\tfrac{3\sqrt{3}}{32}\tfrac{1}{\varepsilon} = \tfrac{1}{A}$.

\textbf{Lower bound (optimality):} Let $\TT$ be any $\varepsilon$-triangulation of $F(x,y) = xy$, with underlying triangulation $\{T_i\}$. For each triangle $T_i$, the restriction of the approximating function to $T_i$ is affine with pointwise error at most $\varepsilon$, so $T_i$ is individually a feasible instance of the single-triangle optimization problem (model~\eqref{model:optimal-triangle}).

The slack-variable reformulation (Section~\ref{sec:slack-reformulation}) transforms each feasible triangle into problem~\eqref{eq:reduced} with $\sigma = 2\varepsilon$. Section~\ref{sec:elem-proof} then establishes the global upper bound $4A^2 \leq \tfrac{1024\sigma^2}{27}$ through the following chain:
\begin{enumerate}
\item Separate convexity reduces to box-vertex maximization (Lemma~\ref{lem:vertex-reduction}).
\item Type~A dominance (Lemma~\ref{lem:tight}) reduces to the 3-variable problem~\eqref{eq:3var}.
\item Symmetry (Theorem~\ref{thm:symmetry}) gives $a_2 = a_3$, reducing to 2 variables.
\item The algebraic certificate (Theorem~\ref{thm:certificate}) proves $4A^2 \leq \tfrac{1024\sigma^2}{27}$ with equality only at $(a^*, b^*) = (\sqrt{\sigma}, \sqrt{\sigma}/3)$.
\item The combined certificate (Theorem~\ref{thm:combined}) verifies this bound holds at all 8 box vertices, hence over the full feasible set.
\end{enumerate}
With $\sigma = 2\varepsilon$, every feasible triangle satisfies $A(T_i) \leq A^* = \tfrac{32\sqrt{3}}{9}\varepsilon$.

Since the triangles $\{T_i\}$ have pairwise disjoint interiors and each has area at most $A^*$, the number of triangles contained in any square $Q_R = [-R,R]^2$ satisfies $|I_R| \geq \operatorname{area}(Q_R)/A^* = 4R^2/A^*$. Triangles meeting the boundary of $Q_R$ contribute a lower-order correction (at most $O(R)$ triangles, since $\TT$ is locally finite), so the triangle density satisfies $\delta(\TT) \geq 1/A^* = \tfrac{3\sqrt{3}}{32}\tfrac{1}{\varepsilon}$.

\textbf{Verification:} We verified this solution by (i) checking that the optimal triangle with $d_1 = -\varepsilon$, $d_2 = d_3 = 7\varepsilon/9$, and $k = x_2 y_2 = 32\varepsilon/9$ satisfies all original constraints of model~\eqref{model:optimal-triangle} with equalities $E_i^* = \pm\varepsilon$; (ii) confirming the area formula $A = \tfrac{32\sqrt{3}}{9}\varepsilon$ using symbolic computation; and (iii) solving the original nonconvex optimization problem numerically with multiple starting points, all converging to the same optimal value.
\end{proof}

\begin{proof}[Proof of \cref{thm:optimal-triangulation-under} (\Pwl underestimations/overestimations)]
\label{thm:optimal-triangulation-under-proof}

\textbf{Upper bound (achievability):} By Lemma~\ref{lem:area_density}, it suffices to exhibit a single feasible triangle. The triangle from Section~\ref{sec:scaling} with $\sigma = \varepsilon$ has area $\tfrac{16\sqrt{3}}{9}\varepsilon$ and satisfies all error constraints. By Lemma~\ref{lem:area_density}, this yields a triangulation with density $\tfrac{3\sqrt{3}}{16}\tfrac{1}{\varepsilon}$.

\textbf{Lower bound (optimality):} Let $\TT$ be any pwl $\varepsilon$-underestimation of $F(x,y) = xy$ (the overestimation case is symmetric). On each triangle $T_i$ of the triangulation, the restriction is an affine function with $f|_{T_i} \leq F$ and $F - f|_{T_i} \leq \varepsilon$, so the deviations satisfy $d_j \in [-\varepsilon, 0]$ for all vertices $j$ of $T_i$. Hence $T_i$ is a feasible instance of the single-triangle problem with $E_{\min} = -\varepsilon$ and $E_{\max} = 0$, giving $\sigma = \varepsilon$.

By Corollary~\ref{cor:optimal} with $\sigma = \varepsilon$, every such feasible triangle satisfies $A(T_i) \leq A^* = \tfrac{16\sqrt{3}}{9}\varepsilon$. The same area-counting argument as in \cref{thm:optimal-triangulation} then gives $\delta(\TT) \geq 1/A^* = \tfrac{3\sqrt{3}}{16}\tfrac{1}{\varepsilon}$.

The optimal parameters are:
\begin{align*}
\text{Underestimation:} \quad & d_1 = -\varepsilon, \; d_2 = d_3 = -\tfrac{\varepsilon}{9}, \; k = \tfrac{16}{9}\varepsilon, \; \lambda = \tfrac{3}{4}.\\
\text{Overestimation:} \quad & d_1 = 0, \; d_2 = d_3 = \tfrac{8}{9}\varepsilon, \; k = \tfrac{16}{9}\varepsilon, \; \lambda = \tfrac{3}{4}.
\end{align*}

The optimal triangle density is $\tfrac{1}{A^*} = \tfrac{3\sqrt{3}}{16}\tfrac{1}{\varepsilon}$.
\end{proof}

\subsection{Continuous approximations}

\begin{proof}[Proof of \cref{thm:optimal-triangulation-cont} (Continuous \pwl approximations --- parallelogram tilings)]
\label{thm:optimal-triangulation-cont-proof}
For continuous approximations, the deviations at shared vertices must be equal across adjacent triangles. We prove the optimal triangle density is $\tfrac{\sqrt{3}}{8\varepsilon}$.

\textbf{Step 1: Continuity forces constant deviation.}

By Lemma~\ref{lem:area_density}, the tiling is constructed from a triangle $T$ with vertices $v_1, v_2, v_3$ (where $v_1 = (0,0)$) and its point-reflection $T'$ with vertices $-v_1, -v_2, -v_3$. To form a parallelogram that tiles the plane, $T'$ is then translated by the vector $v_2 + v_3$ (which equals twice the midpoint between $v_2$ and $v_3$). The resulting parallelogram $P = T \cup T'$ has vertices $v_1 = (0,0)$, $v_2$, $v_3$, and $v_2 + v_3$, and tiles the plane by translation along the vectors $v_2$ and $v_3$. See Figure~\ref{fig:parallelogram-tiling} for an illustration.

\begin{figure}[ht]
\centering
\begin{tikzpicture}[scale=1.2, >=latex,
    triangle/.style={thick, draw=blue!70!black, fill=blue!20, opacity=0.7},
    triangle_reflected/.style={thick, draw=red!70!black, fill=red!20, opacity=0.7},
    vertex/.style={circle, fill=black, inner sep=1.5pt},
    label/.style={font=\small}]
\coordinate (v1) at (0,0);
\coordinate (v2) at (2.5,0.8);
\coordinate (v3) at (0.8,2.2);

\draw[triangle] (v1) -- (v2) -- (v3) -- cycle;
\node[label] at ($(v1)!0.3!(v2) + (0.35,+.9)$) {$T$};

\coordinate (v1r) at (-0, -0);
\coordinate (v2r) at (-2.5, -0.8);
\coordinate (v3r) at (-0.8, -2.2);
\draw[triangle_reflected, dashed] (v1r) -- (v2r) -- (v3r) -- cycle;
\node[label, red!70!black] at ($(v1r)!0.3!(v2r) + (-.5,-.6)$) {$T'$ (reflected)};

\coordinate (shift) at ($(v2) + (v3)$);
\coordinate (v1t) at ($(v1r) + (shift)$);
\coordinate (v2t) at ($(v2r) + (shift)$);
\coordinate (v3t) at ($(v3r) + (shift)$);
\draw[triangle_reflected] (v1t) -- (v2t) -- (v3t) -- cycle;
\node[label, red!70!black] at ($(v1t)!0.3!(v2t) + (-.3,-0.5)$) {$T'$ (shifted)};

\draw[thick, dashed, gray!50] (v1) -- (v2) -- (shift) -- (v3) -- cycle;

\node[vertex, label={[label] below left:$v_1$}] at (v1) {};
\node[vertex, label={[label] right:$v_2$}] at (v2) {};
\node[vertex, label={[label] above:$v_3$}] at (v3) {};
\node[vertex, label={[label] above right:$v_2+v_3$}] at (shift) {};

\draw[->, thick, red!70!black] ($(v1r)!0.5!(v2r)!0.5!(v3r)$) -- ($(v1t)!0.5!(v2t)!0.5!(v3t) + (-.2,-.2)$) node[midway, below, sloped, label] {translate by $v_2+v_3$};

\draw[->, thin, gray] (-3.5,0) -- (4,0) node[right] {$x$};
\draw[->, thin, gray] (0,-3.5) -- (0,4) node[above] {$y$};
\end{tikzpicture}
\caption{Forming a parallelogram from triangle $T$ and its point-reflection $T'$. Triangle $T$ has vertices $v_1 = (0,0)$, $v_2$, and $v_3$. The point-reflection $T'$ (dashed red) has vertices $-v_1$, $-v_2$, $-v_3$. Translating $T'$ by the vector $v_2 + v_3$ forms a parallelogram (solid red) with vertices $v_1$, $v_2$, $v_3$, and $v_2 + v_3$. This parallelogram tiles the plane by translation along $v_2$ and $v_3$.}
\label{fig:parallelogram-tiling}
\end{figure}
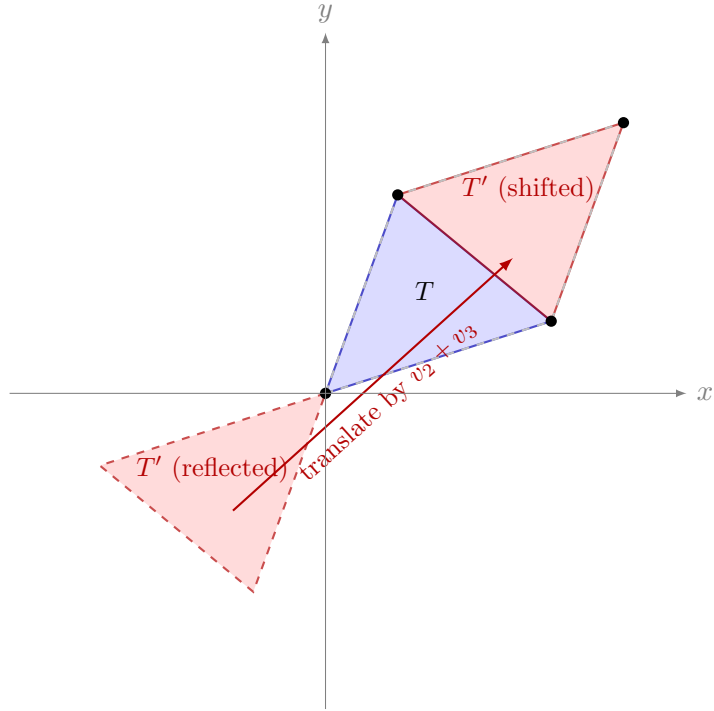

In this tiling, each vertex position is shared by six triangles---three copies of $T$ and three copies of $T'$---meeting in a hexagonal pattern.  For the PWL approximation to be continuous, the function value $f(v) = F(v) + d$ must be the same in all triangles sharing vertex $v$.

We now trace the deviation identifications at each vertex of the fundamental parallelogram.

Recall that the reflected triangle $T'$ (before translation) has vertices $-v_1 = 0$, $-v_2$, $-v_3$, and by Lemma~\ref{lem:invariance_reflection}, the deviation at $-v_i$ in $T'$ equals $d_i$ (the deviation at $v_i$ in $T$). Translating $T'$ by $v_2 + v_3$ preserves deviations (Lemma~\ref{lem:invariance_trans}), yielding the shifted triangle with vertices $v_2 + v_3$ (deviation~$d_1$), $v_3$ (deviation~$d_2$), and $v_2$ (deviation~$d_3$).

\emph{Constraint at $v_2$:} Triangle $T$ assigns deviation $d_2$ at $v_2$. The shifted $T'$ within the same parallelogram assigns deviation $d_3$ at $v_2$. Continuity requires $d_2 = d_3$.

\emph{Constraint at $v_2$ from adjacent parallelogram:} The parallelogram translated by $-v_3$ has its vertex $v_2 + v_3$ at position $v_2 + v_3 - v_3 = v_2$. The shifted $T'$ within that parallelogram assigns deviation $d_1$ at its $v_2 + v_3$ vertex (as derived above). Therefore, continuity at position $v_2$ additionally requires $d_2 = d_1$.

Combining: $d_1 = d_2 = d_3 =: \Delta$.

\textbf{Step 2: Optimization with constant deviation.}
With $d_1 = d_2 = d_3 = \Delta$, the error analysis simplifies. For an ascending edge from $(0,0)$ to $(x_2, y_2)$:

By Lemma~\ref{lem:max_error} with $d_1 = d_2 = \Delta$, we have $\lambda^* = \frac{\Delta - \Delta}{2x_2y_2} + \frac{1}{2} = \frac{1}{2}$.

The interior error is:
\[
E_{interior} = (1-\tfrac{1}{2})\Delta + \tfrac{1}{2}\Delta + \tfrac{1}{2}(1-\tfrac{1}{2})x_2 y_2 = \Delta + \tfrac{1}{4}x_2 y_2.
\]

For the constraint $E_{interior} \leq \varepsilon$: $\Delta + \tfrac{1}{4}x_2 y_2 \leq \varepsilon$, so $x_2 y_2 \leq 4(\varepsilon - \Delta)$.

For the descending edge from $(x_2, y_2)$ to $(x_3, y_3)$ with $(x_2, y_2) = (y_3, x_3)$:
\[
E_{descending} = \Delta - \tfrac{1}{4}(x_2 - y_2)^2.
\]
For the constraint $E_{descending} \geq -\varepsilon$: $\Delta - \tfrac{1}{4}(x_2 - y_2)^2 \geq -\varepsilon$, so $(x_2 - y_2)^2 \leq 4(\varepsilon + \Delta)$.

\textbf{Step 3: Maximizing area.}
The area is $A = \tfrac{1}{2}(x_2^2 - y_2^2) = \tfrac{1}{2}(x_2 - y_2)(x_2 + y_2)$.

Let $p = x_2 + y_2$ and $q = x_2 - y_2$. Then $x_2 y_2 = \tfrac{1}{4}(p^2 - q^2)$.

Constraints become:
\begin{align*}
\tfrac{1}{4}(p^2 - q^2) &\leq 4(\varepsilon - \Delta), \\
q^2 &\leq 4(\varepsilon + \Delta).
\end{align*}

Area: $A = \tfrac{1}{2}pq$.

To maximize $A$, both constraints should be tight. From the second: $q = 2\sqrt{\varepsilon + \Delta}$.

From the first with equality: $p^2 = q^2 + 16(\varepsilon - \Delta) = 4(\varepsilon + \Delta) + 16(\varepsilon - \Delta) = 20\varepsilon - 12\Delta$.

So $p = \sqrt{20\varepsilon - 12\Delta}$ (taking positive root).

Area: $A(\Delta) = \tfrac{1}{2} \cdot 2\sqrt{\varepsilon + \Delta} \cdot \sqrt{20\varepsilon - 12\Delta} = \sqrt{(\varepsilon + \Delta)(20\varepsilon - 12\Delta)}$.

To maximize, take derivative with respect to $\Delta$:
\[
\frac{dA^2}{d\Delta} = (20\varepsilon - 12\Delta) - 12(\varepsilon + \Delta) = 20\varepsilon - 12\Delta - 12\varepsilon - 12\Delta = 8\varepsilon - 24\Delta.
\]
Setting to zero: $\Delta^* = \tfrac{\varepsilon}{3}$.

\textbf{Step 4: Computing optimal area.}
With $\Delta = \tfrac{\varepsilon}{3}$:
\begin{align*}
q &= 2\sqrt{\varepsilon + \tfrac{\varepsilon}{3}} = 2\sqrt{\tfrac{4\varepsilon}{3}} = \tfrac{4}{\sqrt{3}}\sqrt{\varepsilon}, \\
p &= \sqrt{20\varepsilon - 4\varepsilon} = \sqrt{16\varepsilon} = 4\sqrt{\varepsilon}, \\
A^* &= \tfrac{1}{2}pq = \tfrac{1}{2} \cdot 4\sqrt{\varepsilon} \cdot \tfrac{4}{\sqrt{3}}\sqrt{\varepsilon} = \tfrac{8}{\sqrt{3}}\varepsilon = \tfrac{8\sqrt{3}}{3}\varepsilon.
\end{align*}

The optimal triangle density is $\tfrac{1}{A^*} = \tfrac{\sqrt{3}}{8\varepsilon} = \tfrac{3\sqrt{3}}{24}\tfrac{1}{\varepsilon}$.
\end{proof}

The proof above establishes optimality among parallelogram tilings with constant deviation. A natural question is whether non-parallelogram tilings might achieve better density through varying deviations. We conjecture that this is not possible.

\begin{conjecture}[Optimality extends to all continuous triangulations]
\label{conj:continuous-optimality}
Let $\TT$ be any locally finite triangulation of $\R^2$ and let $f$ be a continuous \pwl $\varepsilon$-approximation of $F(x,y) = xy$ with respect to $\TT$. Then the triangle density satisfies $\delta(\TT) \geq \tfrac{\sqrt{3}}{8\varepsilon}$.
\end{conjecture}

The following observation provides partial evidence: the optimal general triangles from \cref{cor:optimal} cannot themselves be assembled into a continuous triangulation.

\begin{remark}[Optimal general triangles cannot tile continuously]
\label{lem:no-optimal-tiling}
We sketch an obstruction showing that there is no locally finite triangulation $\TT$ of $\R^2$ and continuous \pwl $\varepsilon$-approximation $f$ of $F(x,y)=xy$ with respect to $\TT$ such that every triangle $T \in \TT$ achieves the optimal general area $A^*_{\mathrm{gen}} = \tfrac{32\sqrt{3}}{9}\varepsilon$.

Suppose for contradiction that such a triangulation exists. By \cref{cor:optimal}, every triangle must have deviations $d_1 = -\varepsilon$ and $d_2 = d_3 = \tfrac{7\varepsilon}{9}$ (the LHH pattern, up to relabeling). By symmetry, the HLL pattern ($d_1 = \varepsilon$, $d_2 = d_3 = -7\varepsilon/9$) yields the same geometric conclusion. Since $f$ is continuous, each vertex $v$ has a single well-defined deviation $d(v) = f(v) - F(v)$. The vertices partition into two classes:
\[
L = \{v : d(v) = -\varepsilon\} \quad\text{(low)}, \qquad H = \{v : d(v) = \tfrac{7\varepsilon}{9}\} \quad\text{(high)}.
\]
Every triangle has exactly one vertex in $L$ and two in $H$, so $L$ is an independent set (no $L$-$L$ edges).

Now consider the geometric constraints. For each $L$-vertex $u$ at position $(p,q)$, the optimal triangle geometry (\cref{cor:optimal}) requires the two $H$-vertices of each incident triangle to lie on the translated hyperbola
\[
\mathcal{H}_u : (x-p)(y-q) = k^* = \tfrac{32\varepsilon}{9}.
\]
Therefore, \emph{all} $H$-vertices adjacent to $u$ lie on $\mathcal{H}_u$.

Consider two distinct $L$-vertices $u_1$ and $u_2$ at positions $(p_1, q_1)$ and $(p_2, q_2)$, and suppose they share an $H$-vertex $v$. Then $v$ lies on both $\mathcal{H}_{u_1}$ and $\mathcal{H}_{u_2}$. Eliminating $y$ shows that the $x$-coordinates of intersection points satisfy
\begin{equation}
\label{eq:hyperbola-intersection}
(q_1 - q_2)(x - p_1)(x - p_2) = k^*(p_2 - p_1),
\end{equation}
a quadratic equation (when $q_1 \neq q_2$), so $\mathcal{H}_{u_1}$ and $\mathcal{H}_{u_2}$ share at most two points. When $q_1 = q_2$, the equation degenerates to $k^*(p_2 - p_1) = 0$, which has no solution for $p_1 \neq p_2$, meaning the hyperbolas do not intersect at all.

In a planar triangulation, each $L$-vertex $u$ has degree at least~$3$ (and average degree~$6$ by Euler's formula), and its incident triangles form a fan with $H$-vertices $h_1, h_2, \ldots, h_m$ ($m \geq 3$) arranged cyclically on~$\mathcal{H}_u$. Across each $H$-$H$ edge $\{h_i, h_{i+1}\}$, there is an adjacent triangle with a different $L$-vertex~$u'_i$, and $\{h_i, h_{i+1}\} \subseteq \mathcal{H}_{u'_i}$. Since any two distinct hyperbolas $\mathcal{H}_u$ and $\mathcal{H}_{u'}$ share at most two points, these constraints severely restrict the combinatorial structure. We believe a careful counting argument based on the at-most-two-intersection property yields a contradiction, but a complete proof requires a detailed case analysis of the planar graph structure that we leave to future work.
\end{remark}

A complete proof of \cref{conj:continuous-optimality} remains open. The following remark discusses the remaining gap.

\begin{remark}[Remaining gap]
\label{rem:conjecture-evidence}
Remark~\ref{lem:no-optimal-tiling} provides a strong obstruction against assembling a continuous triangulation from optimal general triangles \emph{exclusively}. To prove the full conjecture, one must also rule out triangulations where individual triangles use non-constant deviations to achieve area exceeding $A^*_{\mathrm{const}} = \tfrac{8\sqrt{3}}{3}\varepsilon$, even though the average area cannot.

The key difficulty is an averaging argument. For any triangle with constant deviations $d_1 = d_2 = d_3 = \Delta$, the maximum area is $A(\Delta) = \sqrt{(\varepsilon + \Delta)(20\varepsilon - 12\Delta)}$, maximized at $\Delta^* = \varepsilon/3$ with value $A^*_{\mathrm{const}}$. Triangles can only exceed this bound if their deviations are non-constant, requiring a ``spread'' between the lowest and highest vertex deviations. Continuity constrains this spread globally: a vertex with low deviation must serve as the low vertex in all incident triangles, creating a bottleneck that limits how many triangles can simultaneously benefit.

One approach to close the gap is a charging argument: define a gain function $g(T) = A(T) - A^*_{\mathrm{const}}$ and show that continuity constraints force $\sum_{T \in \TT_R} g(T) \leq o(|\TT_R|)$ for any bounded region~$R$. We leave the development of such an argument to future work.
\end{remark}

\begin{remark}
If \cref{conj:continuous-optimality} holds, then the bound is achieved uniquely by the parallelogram tiling with constant deviation $\Delta = \varepsilon/3$, confirming that Atariah et al.'s construction \cite{Atariah:2018} is optimal among all continuous approximations.
\end{remark}

\begin{remark}[Why computational verification of the conjecture is infeasible]
\label{rem:computational-infeasibility}
Verifying \cref{conj:continuous-optimality} computationally over finite patches is intractable because the optimization landscape is highly nonconvex. For a continuous triangulation over a bounded domain with $n$ vertices, the optimization problem over vertex positions and deviations involves $O(n^2)$ free vertex positions in $\R^2$, subject to continuity constraints and nonlinear error expressions. This leads to a nonconvex global optimization problem with $O(n^2)$ variables and nonlinear constraints, making exhaustive search or even local optimization over moderate problem sizes impractical. Moreover, boundary effects and the lack of translational symmetry in bounded domains mean that small-scale experiments may not generalize to infinite tilings. Therefore, resolving the conjecture fully likely requires a theoretical approach rather than computational verification.
\end{remark}

\begin{remark}[Unified framework recovers all cases]
\label{rem:unified}
The slack-variable framework from Section~\ref{sec:slack-reformulation} provides a unified treatment of all cases. At the Type~A optimum with $a_2 = a_3 =: b$, we have $4A^2 = 16(\sigma - b^2)(\sigma + a^2 + 2ab)$ where $a := a_1$. Setting $c = \sqrt{\sigma}$, this is $16(c^2 - b^2)(c^2 + a^2 + 2ab)$, and specializing to the different approximation types yields:

\textbf{Continuous approximations.} The constraint $d_1 = d_2 = d_3 = \Delta$ translates to $a = b$ (since all slack variables are equal). With this constraint, the objective becomes $16(c^2 - a^2)(c^2 + 3a^2)$, a single-variable optimization:
\[
\frac{d}{da}\left[(c^2 - a^2)(c^2 + 3a^2)\right] = 2a(4c^2 - 12a^2) = 0 \quad\Rightarrow\quad a^* = \frac{c}{\sqrt{3}}.
\]
Then $\Delta^* = \varepsilon - a^{*2} = \varepsilon - c^2/3 = \varepsilon/3$, recovering the result of Atariah et al.\ \cite{Atariah:2018}.

\textbf{Why $\Delta^* = \varepsilon/3$ specifically?} The objective $(c^2 - a^2)(c^2 + 3a^2)$ has the form $X \cdot Y$ where $X = c^2 - a^2$ and $Y = c^2 + 3a^2$. Since $X + Y = 2c^2 + 2a^2$ is not constant in $a$, this is not a direct AM-GM problem. However, the derivative condition $4c^2 = 12a^2$ shows that at the optimum, $Y = 3X$ (i.e., the factors are in ratio 1:3). This yields $a^{*2} = c^2/3 = 2\varepsilon/3$, and hence $\Delta^* = \varepsilon - 2\varepsilon/3 = \varepsilon/3$.

\textbf{Interpolation.} Setting $\Delta = 0$ gives $a = b = \sqrt{\varepsilon}$. With $c^2 = 2\varepsilon$:
\[
A^2 = 4(c^2 - a^2)(c^2 + 3a^2) = 4(\varepsilon)(5\varepsilon) = 20\varepsilon^2, \quad A = 2\sqrt{5}\varepsilon,
\]
recovering the result of Pottmann et al.\ \cite{Pottmann:2000}.

\textbf{Summary.} Continuous approximations correspond to the constraint $a = b$, interpolation to the specific point $a = b = \sqrt{\varepsilon}$, and general approximations to optimization over independent $a, b$ with the optimal ratio $a^*:b^* = 3:1$.
\end{remark}

\subsection{Summary of optimal solutions}

Table~\ref{tab:optimal-solutions} provides the detailed optimal triangle parameters for all approximation types. Combined with Table~\ref{tab:intro-summary} in the introduction, this gives a complete characterization of the optimal solutions.

\begin{table}[ht]
\centering
\renewcommand{\arraystretch}{1.4}
\begin{tabular}{|l|c|c|c|c|c|}
\hline
Method & $k = x_2 y_2$ & $d_1$ & $d_2 = d_3$ & $\lambda_1$ & $\lambda_3$ \\
\hline
General & $\frac{32}{9}\varepsilon$ & $-\varepsilon$ & $\frac{7}{9}\varepsilon$ & $\frac{3}{4}$ & $\frac{1}{2}$ \\
\hline
Continuous & $\frac{8}{3}\varepsilon$ & $\frac{1}{3}\varepsilon$ & $\frac{1}{3}\varepsilon$ & $\frac{1}{2}$ & $\frac{1}{2}$ \\
\hline
Interpolation & $4\varepsilon$ & $0$ & $0$ & $\frac{1}{2}$ & $\frac{1}{2}$ \\
\hline
Overestimation & $\frac{16}{9}\varepsilon$ & $0$ & $\frac{8}{9}\varepsilon$ & $\frac{3}{4}$ & $\frac{1}{2}$ \\
\hline
Cont.\ Overest. & $\frac{4}{3}\varepsilon$ & $\frac{2}{3}\varepsilon$ & $\frac{2}{3}\varepsilon$ & $\frac{1}{2}$ & $\frac{1}{2}$ \\
\hline
Underestimation & $\frac{16}{9}\varepsilon$ & $-\varepsilon$ & $-\frac{1}{9}\varepsilon$ & $\frac{3}{4}$ & $\frac{1}{2}$ \\
\hline
Cont.\ Underest. & $\frac{4}{3}\varepsilon$ & $-\frac{1}{3}\varepsilon$ & $-\frac{1}{3}\varepsilon$ & $\frac{1}{2}$ & $\frac{1}{2}$ \\
\hline
\end{tabular}
\caption{Summary of optimal triangle parameters for different approximation types. The parameter $\lambda_1$ is the convex combination parameter at which the maximum error is attained on each ascending edge (the same for both edges $e_1$ and $e_2$, since $d_2 = d_3$ implies $k_1 = k_2$); $\lambda_3$ is the corresponding parameter for the descending edge $e_3$.}
\label{tab:optimal-solutions}
\end{table}

The table reveals several structural patterns:
\begin{itemize}
\item Non-continuous cases achieve $\lambda_1 = 3/4$, while continuous cases have $\lambda_1 = 1/2$.
\item The ratio $a^*:b^* = 3:1$ (equivalently $\lambda_1 = 3/4$) characterizes all non-continuous optimal solutions.
\item Continuous approximations force $d_1 = d_2 = d_3$, which constrains $\lambda_1 = 1/2$.
\item Under/overestimation have exactly half the area of general approximation in both continuous and non-continuous cases.
\end{itemize}

\begin{remark}[Non-continuity of optimal approximations]
The optimal general approximation from \cref{thm:optimal-triangulation} has $d_1 = -\varepsilon$ and $d_2 = d_3 = \tfrac{7}{9}\varepsilon$. In a parallelogram tiling, each vertex plays the same role (e.g., always ``vertex~1'') in every triangle containing it, so a single well-defined deviation is assigned at each vertex. Since the optimal deviations are not all equal, adjacent triangles sharing a vertex would require different deviation values at that vertex, making the resulting approximation necessarily discontinuous. For general (non-parallelogram) tilings, the situation is more subtle: a vertex could serve as the low-deviation vertex in one incident triangle and a high-deviation vertex in another, potentially allowing continuity. Remark~\ref{lem:no-optimal-tiling} shows that even this more flexible setting cannot produce a continuous tiling from optimal general triangles exclusively, though the full \cref{conj:continuous-optimality} remains open.
\end{remark}  

\section{Conclusion}
\label{sec:conclusion}

We have determined optimal triangulations for piecewise linear approximations of indefinite quadratic functions over the plane, resolving the open question left by Atariah et al.\ \cite{Atariah:2018} regarding whether their constant-deviation construction was optimal. Our main result establishes that the optimal triangle density for general \pwl $\varepsilon$-approximations of $F(x,y) = xy$ is $\tfrac{3\sqrt{3}}{32\varepsilon}$, representing a 25\% reduction compared to their construction. The key insight is that \emph{varying} deviations at vertices---specifically $d_1 = -\varepsilon$ and $d_2 = d_3 = 7\varepsilon/9$---yield strictly larger triangles than the constant deviation $\Delta = \varepsilon/3$ used by Atariah et al. We also proved optimality for several important special cases: continuous approximations among parallelogram tilings (and conjectured this extends to all continuous triangulations), interpolations, and one-sided (under/over) estimations.
 
The key technical contribution is a reformulation via slack variables that transforms the nonlinear optimization into a structured problem over box-constrained edge products, followed by a sequence of exact reductions: separate convexity eliminates the edge-product variables, Type~A dominance reduces to three slack variables, a strict inequality proves symmetry ($a_2 = a_3$), and an algebraic optimality certificate yields the closed-form solution with $a^* = \sqrt{\sigma}$, $b^* = \sqrt{\sigma}/3$. A combined certificate (Theorem~\ref{thm:combined}) verifies global optimality by checking nonnegativity of the gap function at all box vertices. 
 
Our results have implications for computational applications involving bilinear terms in optimization. Compared to regular axis-aligned grids, optimal triangulations achieve 54\% larger triangles in the general case and 131\% larger in the continuous case (see Section~\ref{sec:intro}), suggesting that optimally-oriented triangulations could significantly reduce the computational cost of \pwl approximation schemes for bilinear functions.

Several directions remain open for future work. First, proving or disproving \cref{conj:continuous-optimality}---that the parallelogram tiling optimum extends to all continuous triangulations---would complete our understanding of the continuous case. Second, extending these results to higher dimensions---optimal simplicial approximations in $\R^n$---would be valuable for applications in higher-dimensional optimization. Third, the bounded domain case, where boundary effects may permit lower densities than periodic tilings, remains largely unexplored. Fourth, approximating more general indefinite quadratic forms beyond saddle surfaces (e.g., hyperbolic paraboloids with different principal curvatures) may reveal additional structure.

\bibliographystyle{plain}  
\bibliography{references_updated.bib}

\end{document}